\documentclass{amsart}

\usepackage{graphicx,amssymb,amsmath,xcolor,enumerate}
\usepackage{hyperref}
\hypersetup{
	pdffitwindow=false,     
	pdfstartview={FitH},    
	colorlinks=true,      
	citecolor=red,
}

\usepackage{enumitem}
\usepackage{mathrsfs}
\usepackage{float}

\makeatletter
\def\l@subsection{\@tocline{2}{0pt}{2.5pc}{5pc}{}}
\makeatother

\usepackage{underscore}

\usepackage{geometry}
\DeclareMathOperator{\sgn}{sgn}

\DeclareSymbolFont{largesymbol}{OMX}{yhex}{m}{n}
\DeclareMathAccent{\Widehat}{\mathord}{largesymbol}{"62}
\newcommand*\di{\mathop{}\!\mathrm{d}}

\def\e{\varepsilon}

\numberwithin{equation}{section}              
\newtheorem{theorem}{Theorem}[section]

\newtheorem{lemma}{Lemma}[section]
\newtheorem{proposition}{Proposition}[section]
\newtheorem*{proposition*}{Proposition}

\newtheorem*{corollary*}{Corollary}
\newtheorem{definition}{Definition}[section]
\newtheorem*{definitions*}{Definitions}

\newtheorem*{acknowledgements*}{Acknowledgements}

\newtheorem*{conjecture*}{\bf Conjecture}

\newtheorem*{example*}{\bf Example}
\theoremstyle{remark}
\newtheorem{remark}{\bf Remark}[section]

\begin{document}
\date{}                                     

\author{Yu Gao}
\address[Y. Gao]{Department of Applied Mathematics, The Hong Kong Polytechnic University, Hung Hom, Kowloon, Hong Kong}
\email{mathyu.gao@polyu.edu.hk}

\title[Conservative peakons of mCH]{On conservative sticky peakons to the modified Camassa-Holm equation}

\begin{abstract}
We use a sticky particle method to show global existence of (energy) conservative sticky $N$-peakon solutions to the modified Camassa-Holm equation. A dispersion regularization is provided as a selection principle for the uniqueness of conservative $N$-peakon solutions. The dispersion limit avoids the collision between peakons, and numerical results show that the dispersion limit is exactly the sticky peakons. At last, when the splitting of peakons is allowed, we give an example to show the non-uniqueness of conservative solutions.
\end{abstract}
\maketitle

{\small}

{\small}

\section{Introduction}
In this paper, we are going to study the (energy) conservative peakons of the modified Camassa-Holm (mCH) equation:
\begin{align}\label{eq:mCH}
m_t+[(u^2-u^2_x)m]_x=0,\quad m=u-u_{xx},\quad x\in\mathbb{R},~~t>0.
\end{align}
This equation is an  integrable system \cite{fokas1995class,Fuchssteiner,Olever,Qiao2} with a bi-Hamiltonian structure \cite{GuiLiuOlverQu,Olever} and a Lax-pair \cite{Qiao2}. 
Physically, the mCH equation models the unidirectional propagation of surface waves in shallow water over a flat bottom \cite{fokas1995class}, where $u$ represents the free surface elevation in dimensionless variables.  It was rediscovered from the two-dimensional Euler equations \cite{Qiao2}. The mCH equation was also shown to arise from an intrinsic (arc-length preserving) invariant planar curve flow in Euclidean geometry \cite{GuiLiuOlverQu}. 
Moreover, applying tri-Hamiltonian duality method to the bi-Hamiltonian representation of the KdV and the modified KdV equations leads to the Camassa-Holm (CH) equation 
\begin{align}\label{eq:CH}
m_t+(mu)_x+mu_x=0,\quad m=u-u_{xx},
\end{align}
and the mCH equation \eqref{eq:mCH} respectively; see \cite{fokas1995a,fokas1997plethora,Fuchssteiner,Olever,kang2016liouville}. This is also one of the reason why we call \eqref{eq:mCH} as the modified Camassa-Holm equation. 
This paper concerns the special weak solutions to the mCH equation, which are called peakons and they can be represented by the fundamental solution $G(x)=\frac{1}{2}e^{-|x|}$ to the Helmholtz operator $\mathrm{I}-\partial_{xx}$ in the following form:
\begin{align}\label{eq:Npeakon}
	u^N(x,t)=\sum_{i=1}^Np_i(t)G(x-x_i(t)),~~m^N(x,t)=\sum_{i=1}^Np_i(t) \delta(x-x_i(t)).
\end{align}
These solutions are also referred to as the $N$-peakon (or multi-peakon) solutions. 
Let us compare the mCH equation \eqref{eq:mCH} with the CH equation \eqref{eq:CH} and introduce the weak solutions to the CH equation first to get a better understanding about the special structures for weak solutions to the mCH equation.

Both the CH equation and the mCH equation are integrable systems and they share lots of similarities:

\textbf{$\bullet$}  \emph{Same energy:} The  energy below is conserved for the strong solutions of these two equations:
\begin{align}\label{eq:energy11}
	\mathcal{H}(u)(t):=\int_{\mathbb{R}}(u^2+u_x^2)(x,t)\di x.
\end{align}

\textbf{$\bullet$} \emph{Finite time blow-up behavior:} Both of these two equations possess smooth solutions that develop singularities in finite time; see \cite{mckean1998breakdown,constantin1998global,constantin1998wave,constantin2000blow,brandolese2014local} for the CH equation and \cite{GuiLiuOlverQu,chen2015oscillation,LiuOlver} for the mCH equation.

\textbf{$\bullet$} \emph{Solitons as weak solutions:} Comparing with smooth solitons for the KdV equation and the modified KdV equation, the solitons given by \eqref{eq:Npeakon} for the CH equation and the mCH equation are not smooth, which are called peakons. Moreover, the peakons are weak solutions to these two equations.

\

The CH equation and the mCH equation also have lots of discrepancies:

\textbf{$\bullet$} \emph{Lifespan:}
The lifespan for strong solutions to the CH equation \eqref{eq:CH} with initial data $\e u_0$ is given by $T_{max}\geq C/\e$ for some constant $C=C(u_0)$; see  \cite{danchin2001few,molinet2004well}.
While for the mCH equation \eqref{eq:mCH}, there holds:
$\tilde{T}_{max}\geq \tilde{C}/\e^2$ for some constant $\tilde{C}=\tilde{C}(u_0)$; see \cite[Eq. (12)]{gao2018modified}.

\textbf{$\bullet$} \emph{Blow-up behaviors:}
For the CH equation, the strong solutions blow up via a process that captures the features of breaking waves \cite{whitham2011linear,constantin1998global,constantin1998wave,constantin2000blow}: the solution $u$ remains bounded, but its slope $u_x$ becomes unbounded. For the mCH equation at the blow-up time, the solution $u$ and its derivative $u_x$ remain bounded, but its second order derivative $u_{xx}$ becomes unbounded \cite{GuiLiuOlverQu,chen2015oscillation,gao2018modified}. 

\textbf{$\bullet$} \emph{Formation of peakons:}  For the CH equation, since $u_x$ goes to negative infinity at the blow-up time, peakons cannot be formed because a peakon solution satisfies $u_x^N(\cdot,t)\in L^\infty(\mathbb{R})\cap BV(\mathbb{R})$. However, peakons can be formed for the mCH equation at the blow-up time; see \cite[Theorem 1.3]{gao2018modified}.

One of the most interesting issues, both physically and mathematically, is the question of the continuation of solutions after blow-up time. The well-posedness of weak solutions to the CH equation \eqref{eq:CH} is well studied in the last two decades. For the $N$-peakon solutions (defined by \eqref{eq:Npeakon}) to the CH equation, $\{(p_i(t),x_i(t))\}_{i=1}^N$ evolve by a Hamiltonian system; see system \eqref{eq:CHpeakons}.
When $\{p_i(0)\}_{i=1}^N$ have the same sign, global solutions $\{(x_i(t),p_i(t))\}_{i=1}^N$ can be derived due to the non-collision between peakons; see \cite{CamassaLee,holden2006,AlinaLiu}. In this case,  the Hamiltonian of the $N$-peakon solution to the CH equation  is a constant, which is the same as the energy $\mathcal{H}(u^N)(t)$ (defined by \eqref{eq:energy11}) of the $N$-peakon solution. 
For more general initial data $(u_0,m_0)$, if $u_0\in H^1(\mathbb{R})$ and $m_0=u_0-u_{0xx}\in\mathcal{M}_+(\mathbb{R})$ (finite nonnegative Radon measure space), global existence and uniqueness of weak solutions $u\in C([0,\infty); H^1(\mathbb{R}))$ and $m\in C([0,\infty); \mathcal{M}_+(\mathbb{R}))$ to the CH equation were obtained \cite{Constantin,AlinaLiu,holden2006}. Note that the above initial data essentially prevent the formation of singularities, and the total energy  of the corresponding weak solution is conserved. Indeed, when singularity happens, the solution $u$ remains in $H^1(\mathbb{R})$, but  $m$ is no longer a Radon measure. Hence, it is nature to study the weak solution with an initial datum $u_0\in H^1(\mathbb{R})$ and without the restriction $m_0\in\mathcal{M}_+(\mathbb{R})$. The first result for global existence of weak solutions to the CH equation \eqref{eq:CH} in $H^1(\mathbb{R})$ was obtained in \cite{Xin} by the vanishing viscosity approach. The vanishing viscosity approach singles out the (energy) dissipative solutions, and the uniqueness of their solutions is unknown except for the case when the initial datum $m_0\in\mathcal{M}_+(\mathbb{R})$; see  \cite{xin2002uniqueness}. Several years later, (energy) conservative (weak) solutions to the CH equation was constructed in \cite{Bressan2,holden2007global1,bressan2005optimal}. The key observation is that part of the energy is transferred into singular measures at the blow-up time. 
If one only considers the solution $u$ to the CH equation, there might be some energy dissipation and $u\notin C([0,\infty);H^1(\mathbb{R}))$. 
To describe the conservation of the total energy, another energy measure variable $\mu$ was introduced, and global conservative solutions were obtained with $\di\mu_{ac}(t)=(u^2+u_x^2)\di x$, where $\mu_{ac}(t)$ is the absolutely continuous part of $\mu$ with respect to the Lebesgue measure; see \cite{Bressan2,holden2007global1}. 
Similar methods can also be applied to obtain the dissipative  solutions; see \cite{bressan2007global,holden2009dissipative}. The uniqueness of weak solutions was obtained by the method of characteristics; see \cite{bressan35uniqueness,bressan2016uniqueness} for the uniqueness of the conservative solutions, and see \cite{jamroz2016uniqueness} for the uniqueness of the dissipative solutions.
The above generalized framework of the CH equation were also used to obtain global conservative or dissipative $N$-peakon solutions 
\cite{holden2007global,holden2008global}.

Comparing with the CH equation \eqref{eq:CH}, the well-posedness theory of the weak solutions to the mCH equation \eqref{eq:mCH} is far from satisfactory.  The local well-posedness of strong solutions to the mCH equation was well studied in \cite{FuGui,Himonas,gao2018modified}.   The weak solutions to the mCH equation \eqref{eq:mCH} were first studied in \cite{GuiLiuOlverQu}, where  global  one peakon weak solutions were constructed in the form: $u(x,t)=pe^{-|x-2p^2t/3|}$ for $p\in\mathbb{R}$; see  \cite[Theorem 6.1]{GuiLiuOlverQu}. The peakon is a weak solution  to the mCH equation \eqref{eq:mCH} satisfying (see \cite[Definition 3.1]{GuiLiuOlverQu})
\begin{equation}\label{eq:weaksolution1}
\begin{aligned}
\int_0^T\int_{\mathbb{R}}u(\varphi_t-\varphi_{txx})\di x\di t -\frac{1}{3}\int_0^T\int_{\mathbb{R}}u^3_x\varphi_{xx}\di x\di t-\frac{1}{3}\int_0^T\int_{\mathbb{R}}u^3\varphi_{xxx}\di x\di t\\
+\int_0^T\int_{\mathbb{R}}(u^3+uu_x^2)\varphi_x\di x\di t+\int_{\mathbb{R}}u_0[\varphi(x,0)-\varphi_{xx}(x,0)]\di x=0
\end{aligned}
\end{equation}
for any $\varphi\in C_c^\infty(\mathbb{R}\times[0,T))$. The $N$-peakon solutions were also studied in the same paper \cite{GuiLiuOlverQu} by taking \eqref{eq:Npeakon} into the definition of weak solutions \eqref{eq:weaksolution1}, and some calculations show that $\{p_i\}_{i=1}^N$ are constants independent of time $t$. Moreover, the following ODE system for the trajectories $\{x_i(t)\}_{i=1}^N$ holds before collision (i.e., $x_1(t)<x_2(t)<\cdots<x_N(t)$) \cite[Eq. (6.24)]{GuiLiuOlverQu}:
\begin{align}\label{eq:ODE1}
\frac{\di}{\di t}x_i(t)=\frac{1}{6}p_i^2+\frac{1}{2}\sum_{j<i}p_ip_je^{x_j-x_i}+\frac{1}{2}\sum_{j>i}p_ip_je^{x_i-x_j}+\sum_{1\leq m<i<n\leq N}p_mp_ne^{x_m-x_n}.
\end{align}
Notice that the peakons given by \eqref{eq:ODE1} might collide with each other in finite time even if $\{p_i\}_{i=1}^N$ have the same sign (see \cite[Proposition 4.5]{GaoLiu}), which is different from the CH peakons. After the collision of peakons, two different methods were provided recently to extend the trajectories of peakon solutions globally, i.e., a sticky particle method \cite{GaoLiu} and a double mollification method \cite{gao2018dispersive}. Both of the above two methods generate global $N$-peakon weak solutions in the sense of \eqref{eq:weaksolution1}. Moreover, for any initial datum $m_0
\in \mathcal{M}(\mathbb{R})$, a global weak solution to the mCH equation \eqref{eq:mCH} can be constructed from the $N$-peakon solutions via a mean field limit method \cite{GaoLiu}. The uniqueness of the weak solutions in the sense of \eqref{eq:weaksolution1} is not proved.
As mentioned above, the conservative or dissipative weak solutions to the CH equation are not continuous in time with values in $H^1(\mathbb{R})$ due to potential changing of energy into singular measures. For the weak solution $u$ to the mCH equation, no energy is transferred into singular measures and $u\in C([0,\infty);H^1(\mathbb{R}))$; see \cite[Theorem 3.4]{GaoLiu}.
However, the energy $\mathcal{H}(u^N)(t)$ is not conservative or dissipative for the $N$-peakon solutions constructed by the sticky particle method or the double mollification method, and these weak solutions are not unique; see \cite[Remark 2.3]{gao2018dispersive} and \cite[Proposition 4.5]{GaoLiu}.  
Energy conservation or dissipation is essential to establish the uniqueness of weak solutions for the other similar integrable systems such as the CH equation  \cite{bressan2005optimal,bressan2016uniqueness,bressan35uniqueness,jamroz2016uniqueness}, the Hunter-Saxton equation \cite{cieslak2016maximal,dafermos2011generalized,zhang2000existence,gao2021regularity}, and the Novikov equation \cite{chen2015existence}. 
Hence, the lack of relation between weak solutions and the energy  might be the main difficulty for the problem of uniqueness of weak solutions to the mCH equation in the sense of \eqref{eq:weaksolution1}. 

Recently, another different class of $N$-peakon solutions to the mCH equation was proposed in \cite{chang2016lax,chang2017liouville,chang2018lax}, where the $N$-peakon  solutions are defined via the following ODE system:
\begin{align}\label{eq:ODE}
	\frac{\di}{\di t}x_i(t)=\frac{1}{2}\sum_{j<i}p_ip_je^{x_j-x_i}+\frac{1}{2}\sum_{j>i}p_ip_je^{x_i-x_j}+\sum_{1\leq m<i<n\leq N}p_mp_ne^{x_m-x_n}
\end{align}
for $i=1,2,\cdots,N$ and $x_1(t)<x_2(t)<\ldots<x_N(t)$. Although the $N$-peakon solutions obtained by \eqref{eq:ODE} are not weak solutions in the sense of \eqref{eq:weaksolution1}, they are compatible with the Lax integrability \cite{chang2016lax,chang2018lax}. Moreover, the energy $\mathcal{H}(u^N)(t)$ is conserved; see \cite[Theorem 2.1]{chang2016lax}. We will provide another more algebraic way to prove the energy conservation in this paper; see Proposition \ref{pro:energyconservation} and Lemma \ref{lmm:Npeakonlemmano collision} below. The conservative $N$-peakon solutions also form a Liouville integrable Hamiltonian system; see \cite{chang2017liouville}. However, the global $N$-peakon solutions in the sense of \eqref{eq:ODE} were only obtained when $p_i>0$ ($i=1,2,\cdots,N$) and together with some other conditions; see \cite[Theorem 5.1, 6.1]{chang2018lax}.  The general global existence and uniqueness of conservative  solutions  to the mCH equation is still an open question.

In this paper, we are going to focus on global existence and uniqueness of conservative $N$-peakon solutions to the mCH equation \eqref{eq:mCH} with arbitrary $p_i\in\mathbb{R}$. 
Since the conservative peakons given by \eqref{eq:ODE} are not weak solutions in the sense of \eqref{eq:weaksolution1}, we first need to get a new definition for weak solutions. According to the behavior of strong solutions at the blow-up time \cite[Theorem 5.2]{gao2018modified}, it is nature to study the conservative solutions to the mCH equation with an initial datum $u_0$ in the space $H^1(\mathbb{R})$ and $u_0,~~u_{0x}\in BV(\mathbb{R})$, $m_0=u_0-u_{0xx}\in\mathcal{M}(\mathbb{R})$. Under this setting, the term $mu_x^2$ (or more precisely: $u_x^2u_{xx}$) in the mCH equation becomes a singular product between a BV function $u_x^2$ and a Radon measure $m$. To illustrate the idea for this product, we start from $N$-peakon solutions.
For an $N$-peakon solution $u^N$, the trajectories should satisfy the following ODE system in some sense:
\[
\frac{\di}{\di t}x_i(t)=[(u^N)^2-(u_x^N)^2](x_i(t),t),\quad i=1,2,\cdots,N.
\]
Since the derivative $u_x^N(\cdot,t)$ is a BV function with jump discontinuities at $x_i(t)$, $i=1,2,\cdots,N$, the value of $(u_x^N)^2(x_i(t),t)$ is vague.
Before collision of peakons, direct calculations show that system \eqref{eq:ODE1} corresponds to assigning value 
\[
\frac{1}{3}\left[\overline{\left(u_x^N\right)^2}+2\left(\overline{u_x^N}\right)^2\right](x_i(t),t)
\] 
to $(u^N_x)^2(x_i(t),t)$, while system \eqref{eq:ODE} is to define $(u^N_x)^2(x_i(t),t)$  by $\overline{(u_x^N)^2}(x_i(t),t)$; also see  \cite{chang2016lax,chang2018lax}. Here, for a BV function $f:\mathbb{R}\to\mathbb{R}$, we use $\bar{f}(x)$ to denote the average of the right limit $f(x+)$ and the left limit $f(x-)$, i.e., 
\[
\bar{f}(x)=\frac{1}{2}[f(x+)+f(x-)].
\] 
Hence, to be compatible with \eqref{eq:ODE}, we give a new definition of weak solutions to the mCH equation:
\begin{definition}\label{def:weaksolution}
For $u_0\in H^1(\mathbb{R}),~~u_0-u_{0xx}=m_0\in\mathcal{M}(\mathbb{R})$, a function
\[
u\in C([0,T);H^1(\mathbb{R}))\cap L^\infty(0,T; W^{1,\infty}(\mathbb{R}))\cap W^{1,\infty}(0,T; L^{\infty}(\mathbb{R}))
\]
is said to be a  weak solution of the mCH equation if
\begin{align}\label{def:weak2}
\mathcal{L}(u,\varphi)+\int_{\mathbb{R}}\varphi(x,0)m_0(\di x)=0
\end{align}
holds for all $\varphi\in C_c^\infty(\mathbb{R}\times[0, T))$. Here 
\begin{equation}\label{eq:weakfunctional}
\begin{aligned}
\mathcal{L}(u,\varphi):=
&\int_0^T\int_{\mathbb{R}}u(\varphi_t-\varphi_{txx})\di x\di t+\int_0^T\int_{\mathbb{R}}(u^3+2uu_x^2)\varphi_x\di x\di t\\
&-\frac{1}{3}\int_0^T\int_{\mathbb{R}}u^3\varphi_{xxx}\di x\di t-\int_0^T\int_{\mathbb{R}}\varphi_{x}\overline{u_x^2}m(\di x\di t).
\end{aligned}
\end{equation}
If $T=+\infty$, we call $u$ as a global weak solution to the mCH equation.
\end{definition}
We need to be careful with the last integral in \eqref{eq:weakfunctional}, since it contains a singular product term between a BV function $\overline{u_x^2}$ and a Radon measure $m$. This kind of definition for weak solutions was used in \cite{zheng1994existence}, where Zheng and Majda studied measure weak solutions to one-component Vlasov-Poisson and Fokker-Planck-Poisson systems. They used Vol'pert's calculus about BV functions in \cite{vol1967spaces,volpert1985analysis}  to study the singular product between a BV function and a Radon measure. For an $N$-peakon solution $u^N$, the  last integral in \eqref{eq:weakfunctional} satisfies:
\begin{multline*}
	-\int_0^T\int_{\mathbb{R}}u^N (u_x^N)^2(x,t)\varphi_{x}(x,t)\di x\di t+ \int_0^T\int_{\mathbb{R}}\overline{(u^N_x)^2}(x,t)\varphi_{x}(x,t)u^N_{xx}(\di x\di t)\\
	=-\int_0^T\sum_{i=1}^Np_i\overline{(u^N_x)^2}(x_i(t),t)\varphi_{x}(x_i(t),t)\di t.
\end{multline*}
When $x_1(t)<x_2(t)<\cdots<x_N(t)$, substituting $u^N(x,t)=\sum_{i=1}^Np_iG(x-x_i(t))$ into \eqref{def:weak2} and taking integration by parts, we find the traveling speeds of $\{x_i(t)\}_{i=1}^N$ are exactly the same as the one given by \eqref{eq:ODE}. Hence, Definition \ref{def:weaksolution} is reasonable. Moreover, we can use \eqref{eq:ODE} to obtain $N$-peakon  solutions in the sense of Definition \ref{def:weaksolution} before collision.  However, the trajectories given by \eqref{eq:ODE} might collide with each other in finite time even if $p_i>0$, $i=1,\cdots,N$; see Figure \ref{fig:threetrajectories1}. This is totally different with the $N$-peakon solutions to the CH equation. 
There are two considerable challenges in exploiting the questions for global existence. The first one is how to extend the trajectories globally after collision. The second one is how to guarantee the energy conservation after collision. 
We are going to handle these questions by using the sticky particle method proposed in our previous paper \cite{GaoLiu}. The main idea is as follows. Assume the peakons stick together and combine their momentum $p_i$ whenever they collide. Then, view the collision time as a new starting point and use system \eqref{eq:ODE} again to extend the trajectories. Repeat this process whenever collision happens. Because the trajectories can collide at most $N-1$ times, we can construct global sticky trajectories $\{x_i(t)\}^N_{i=1}$. Then, we use the global sticky trajectories to construct a global sticky peakon weak solution, and the sticky peakon weak solutions turn out to be energy conservative. See details in Section \ref{sec:sticky}.

After obtaining global conservative $N$-peakon solutions, our next goal is to provide a selection principle for the uniqueness of conservative $N$-peakon solutions.
As we mentioned for the other integrable systems, energy conservation or dissipation is very important for the uniqueness of weak solutions. Therefore, after we obtain conservative $N$-peakon solutions to the mCH equation, it is nature to hope for the uniqueness. However, the uniqueness is relatively difficult as explained below.
For the conservative solutions to other integrable systems (e.g., the CH equation and the Hunter-Saxton equation), one can apply the generalized characteristics method to show the uniqueness; see \cite{bressan35uniqueness,gao2021regularity}. The equation for the energy density plays a crucial role. Take the CH equation \eqref{eq:CH} as an example. The conservative solutions to the CH equation are also weak solutions to the following equation for the energy density $u^2+u_x^2$ of the CH equation:
\begin{align}\label{eq:energyeq}
(u^2+u_x^2)_t+[u(u^2+u_x^2)]_x=[u(u^2-2P)]_x, \quad P=G\ast\left(u^2+\frac{u_x^2}{2}\right).
\end{align}
Without the constrain of the energy conservation equation \eqref{eq:energyeq}, it is impossible to obtain the uniqueness of general conservative solutions to the CH equation; see \cite[Theorem 3]{bressan2005optimal}. For the mCH equation \eqref{eq:mCH}, one can also formally derive the following energy conservation equation similar to \eqref{eq:energyeq}:
\[
(u^2+u_x^2)_t+\left(\frac{1}{6}u^4+u^2u_x^2-\frac{1}{2}u_x^4\right)_x+\frac{2}{3}\left[uG_x\ast u_x^3+uG\ast\left(\frac{2}{3}u^3+uu_x^2\right)\right]_x=0.
\]
However, in the process to obtain the above equation, the relation $u_x^2u_{xx}=(u_x^3)_x/3$ was used, which is correct for smooth solutions. But for a weak solution $u$, the derivative $u_x$ is a BV function, and according to Vol'pert's calculus for BV functions \cite{vol1967spaces,volpert1985analysis}, we have
\[
(u_x^3)_x=\left[\overline{u_x^2} +2(\overline{u_x})^2\right]u_{xx},
\]
which is not compatible with  weak solutions in the sense of Definition \ref{def:weaksolution}. Hence, it is difficult to apply the generalized characteristics method to show the uniqueness of conservative solutions to the mCH equation.  
In this paper, we are going to provide a potential method for the uniqueness of conservative $N$-peakon solutions via a dispersion limit (see \eqref{eq:approximateODE}), which is an approximation of the mCH equation in the Lagrangian coordinates. The dispersive regularization will also be referred to as particle blob method in this paper, because it shares some similarities with the traditional vortex blob method for the incompressible Euler equation. It is well known that the solutions to the incompressible Euler equation can be approximated by a finite superposition of interacting elementary solutions, and this method is known as point-vortex method or vortex blob method; see \cite{chorin1973discretization,anderson1985vortex,cottet2000vortex,majda2002vorticity,hauray2009wasserstein}. The methods are based on the Lagrangian equation for particle trajectories of the  incompressible Euler equation.  We are going to show that the dispersive regularization prevent collisions between peakons; see Proposition \ref{pro:nocollisionepsilon}. Numerical simulations show that the dispersive limit is exactly the sticky particles; see Figure \ref{fig:dispersivelimit} and Figure \ref{fig:Four}. Hence,   the dispersive regularization method gives a potential selection principle for the conservative $N$-peakon solutions. For the rigorous proof however,  we can only prove that the limit of the  regularization are the sticky trajectories for $N=3$; see Theorem \ref{thm:dispersionlimit}  and Proposition \ref{pro:dispersionlimit}.  Moreover, if the splitting of peakons is allowed, then we can provide an example to show non-uniqueness of peakons; see Section \ref{sec:splitting}.

The rest of this paper is organized as follows. In Section \ref{sec:algebraic}, some algebraic properties of the vector field of system \eqref{eq:ODE} will be provided. The details for the construction of global conservative sticky $N$-peakon solutions will be given in Section \ref{sec:sticky}. In Section \ref{sec:dispersion}, we will introduce the dispersive regularization (or the particle blob method) for a potential selection principle of uniqueness. We will provide an example for the non-uniqueness of peakon solutions in Section \ref{sec:splitting} when the splitting of peakons is allowed.

\section{Conservative sticky peakons}
\subsection{ODE system and some algebraic properties}\label{sec:algebraic}
First, we are going to show some algebraic properties for the vector field of  system \eqref{eq:ODE}, which are important for the energy conservation of peakon solutions to the mCH equation. 
We have the following results.
\begin{proposition}\label{pro:energyconservation}
Let $p_i\in\mathbb{R}$, $i=1,2,\cdots,N$, and $x_1<x_2<\cdots<x_N$ be some constants. Define
\[
a_{ij}=\frac{1}{2}p_ip_je^{-|x_i-x_j|}, \qquad A^N_k=\sum_{j\neq k}a_{jk}+2\sum_{1\leq m<k<n\leq N}a_{mn},\quad i,j,k=1,2\cdots,N.
\]
Then, for $2\leq N\in\mathbb{N}$ we have
\begin{equation}\label{eq:peakonspeedrelation}
\sum_{j=1}^N(-1)^{j+1}A^N_j=0,
\end{equation}
and
\begin{equation}\label{eq:energy}
\sum_{1\leq i<j\leq N}a_{ij}(A^N_i-A^N_j)=0.
\end{equation}

\end{proposition}
\begin{proof}
We first prove that \eqref{eq:peakonspeedrelation} holds for $N=2$. Actually, we have
$
A^2_1-A^2_2=a_{21}-a_{12}=0.
$
Next, we use induction for $N$ to prove \eqref{eq:peakonspeedrelation}. Assume that it holds for all $N$ satisfying $2\leq N\leq M$. Then, for $N=M+1$, we have
\begin{align*}
A^N_k=A^M_k+a_{kN}+2\sum_{j<k}a_{jN},\quad k\leq M,
\end{align*}
which implies for any $i\leq M$ that
\begin{align}\label{eq:id1}
\sum_{j=1}^i(-1)^{j+1}A^N_j=\sum_{j=1}^i(-1)^{j+1}A^M_j+(-1)^{i+1}\sum_{j=1}^ia_{jN},
\end{align}
and
\begin{align}\label{eq:id2}
\sum_{j=1}^{N-i}(-1)^{j+1}A_{N-j+1}^N=\sum_{j=1}^{M-i}(-1)^{j}A_{M-j+1}^M+(-1)^{N-i+1}\sum_{j=1}^{i}a_{jN}.
\end{align}
Set $i=M$ in the identity \eqref{eq:id1}, and by the assumption for induction we obtain
\begin{align*}
\sum_{j=1}^M(-1)^{j+1}A^N_j=\sum_{j=1}^M(-1)^{j+1}A^M_j+(-1)^{M+1}\sum_{j=1}^Ma_{jN}=(-1)^{M+1}A_{N}^N.
\end{align*}
Therefore,
\[
\sum_{j=1}^{N}(-1)^{j+1}A^N_j=\sum_{j=1}^M(-1)^{j+1}A^N_j+(-1)^{M+2}A_{N}^N=(-1)^{M+1}A_{N}^N+(-1)^{M+2}A_{N}^N=0.
\]

Next, we prove \eqref{eq:energy}. Notice the following identity:
\begin{align}\label{eq:energy1}
\sum_{1\leq i<j\leq N}a_{ij}(A^N_i-A^N_j)=\sum_{k=1}^NA^N_k\left(\sum_{j>k}a_{kj}-\sum_{i<k}a_{ik}\right).
\end{align}
We claim:
\begin{align}\label{eq:energy2}
\sum_{j>k}a_{kj}-\sum_{i<k}a_{ik}=\sum_{j=1}^{k-1}(-1)^jA_{k-j}^N+\sum_{j=1}^{N-k}(-1)^{j+1}A^N_{k+j},\qquad k=1,2,\cdots,N.
\end{align}
Combining \eqref{eq:energy1} and \eqref{eq:energy2} gives \eqref{eq:energy}.

Next, we use induction to show the above claim \eqref{eq:energy2}. It obviously holds for $N=2$. Assume \eqref{eq:energy2} holds for $2\leq N\leq M$. Due to \eqref{eq:id1} and \eqref{eq:id2}, for $N=M+1$ and $k\leq M$, we have
\begin{equation}
\begin{aligned}
&\sum_{j=1}^{k-1}(-1)^jA_{k-j}^N=\sum_{j=1}^{k-1}(-1)^{k-j}A_{j}^N=\sum_{j=1}^{k-1}(-1)^{k-j}(-1)^{j+1}(-1)^{j+1}A_{j}^N\\
=&(-1)^{k+1}\sum_{j=1}^{k-1}(-1)^{j+1}A_{j}^N=(-1)^{k+1}\left(\sum_{j=1}^{k-1}(-1)^{j+1}A^M_j+(-1)^{k}\sum_{j=1}^{k-1}a_{jN}\right)\\
=&\sum_{j=1}^{k-1}(-1)^jA_{k-j}^M-\sum_{j=1}^{k-1}a_{jN},
\end{aligned}
\end{equation}
and
\begin{equation}
\begin{aligned}
&\sum_{j=1}^{N-k}(-1)^{j+1}A^N_{k+j}=\sum_{j=1}^{N-k}(-1)^{N-k-j}A^N_{N-j+1}\\
=&\sum_{j=1}^{N-k}(-1)^{N-k-j}(-1)^{j+1}(-1)^{j+1}A^N_{N-j+1}=(-1)^{N-k+1}\sum_{j=1}^{N-k} (-1)^{j+1}A^N_{N-j+1}\\
=&(-1)^{N-k+1}\left(\sum_{j=1}^{M-k}(-1)^{j}A_{M-j+1}^M+(-1)^{N-k+1}\sum_{j=1}^{k}a_{jN}\right)\\
=&\sum_{j=1}^{M-k}(-1)^{j+1}A_{k+j}^M+\sum_{j=1}^{k}a_{jN}.
\end{aligned}
\end{equation}
Combining the above two identity, we obtain
\begin{multline*}
\sum_{j=1}^{k-1}(-1)^jA_{k-j}^N+\sum_{j=1}^{N-k}(-1)^{j+1}A^N_{k+j}=\sum_{j=1}^{k-1}(-1)^jA_{k-j}^M+\sum_{j=1}^{M-k}(-1)^{j+1}A_{k+j}^M+a_{kN}\\
=\sum_{M\geq j>k}a_{kj}-\sum_{i<k}a_{ik}+a_{kN}=\sum_{N\geq j>k}a_{kj}-\sum_{i<k}a_{ik}.
\end{multline*}

\end{proof}

\subsection{Conservative sticky $N$-peakon solutions}\label{sec:sticky}
As mentioned in \cite{chang2016lax,chang2017liouville,chang2018lax}, $N$-peakon solutions to the mCH equation can be constructed via system \eqref{eq:ODE} only before the collision between peakon trajectories. However, trajectories given by \eqref{eq:ODE} can collide with each other in finite time; see the following numerical results (by the standard Runge-Kutta methods for ODEs).
\begin{figure}[H]
\begin{center}
\includegraphics[width=0.7\textwidth]{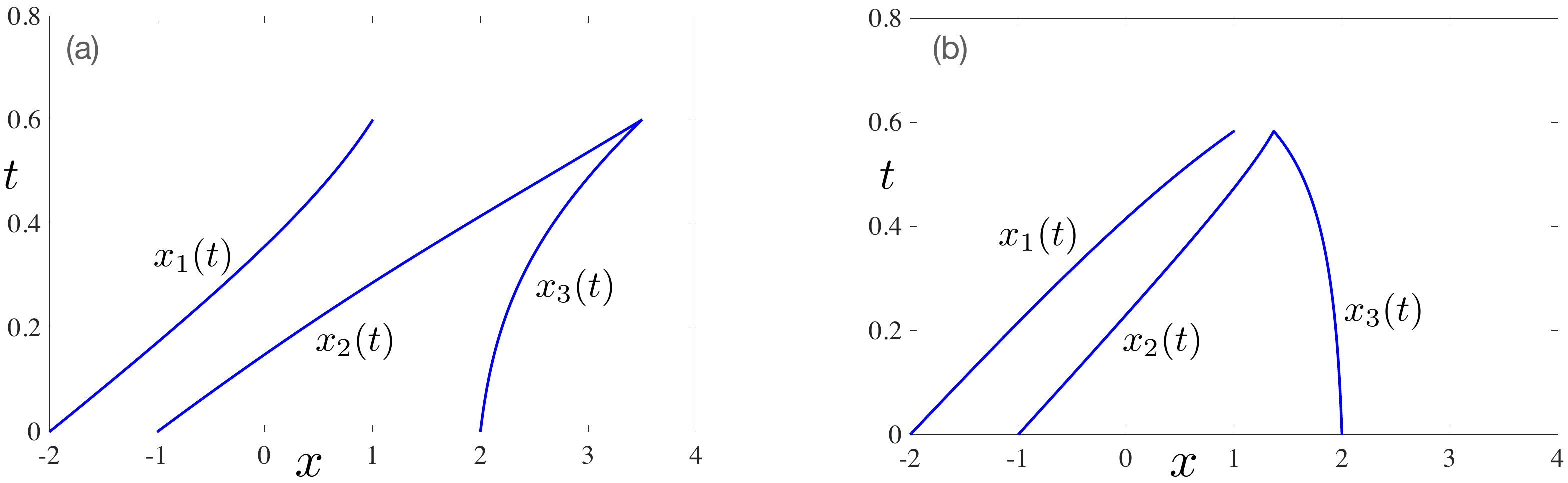}
\end{center}
\caption{Trajectories obtained by ODE system \eqref{eq:ODE} with momentum: (a) $p_1=15,~p_2=2,~p_3=3$;  (b) $p_1=5,~p_2=5,~p_3=-1$.}
\label{fig:threetrajectories1}
\end{figure}
Next, we are going to introduce the sticky particle method to overcome the difficulty brought by collision. This method was first used in \cite{GaoLiu} for system \eqref{eq:ODE1}, where the energy of the sticky peakons is not conserved. Here we will apply the sticky particle method for \eqref{eq:ODE} to obtain global energy conservative sticky peakons. 

For $N\in\mathbb{N}$, consider the following initial datum:
\begin{align}\label{eq:m0N}
m_0^N(x)=\sum_{i=1}^Np_i\delta(x-c_i),\quad c_1<c_2<\cdots<c_N,\quad \sum_{i=1}^N|p_i|\leq M_0.
\end{align}
Here, $p_i$, $c_i$ and $M_0$ are constants.
We have the following lemma:
\begin{lemma}\label{lmm:Npeakonlemmano collision}
Let $m^N_0$ given by \eqref{eq:m0N}. Then, the following statements hold:
\begin{enumerate}
\item [(i)] There is a unique global solution $\{x_i(t)\}_{i=1}^N$ to \eqref{eq:ODE} subject to $\{x_i(0)=c_i\}_{i=1}^N$.
\item [(ii)] Assume that the first  collision time of the  solutions  is $t_1\in \mathbb{R}_+\cup \{+\infty\}$ which means
\[
t_1:=\sup\{t_0:x_1(t)<x_2(t)<\ldots<x_{N}(t) ~\textrm{ for }~t\in [0,t_0)\}.
\]
Then, we have
\begin{align}\label{equilcontinuous}
\left|\frac{\di}{\di t}x_i(t)\right|\leq \frac{1}{2}M_0^2, \quad t\in [0,t_1), \quad i=1,\ldots, N.
\end{align}
\item [(iii)] There is a weak solution to the mCH equation in the sense of Definition \ref{def:weaksolution} subject to initial data $m^N_0$ in $[0,t_1)$, which is given by
\begin{align}\label{eq:Npeakonsoluiton}
u^N(x,t):=\sum_{i=1}^Np_iG(x-x_i(t)),\quad t\in[0,t_1).
\end{align}
\item [(iv)] The energy $\mathcal{H}$ of $u^N$ (defined by \eqref{eq:energy11}) is conserved for $t\in [0,t_1]$.
\end{enumerate}
\end{lemma}

\begin{proof}
The statement (i) is obvious and we only prove (ii), (iii) and (iv).

\textbf{Step 1}.   Proof of (ii).
By the definition of $t_1$, we know $x_1(t)<x_2(t)<\ldots<x_N(t)$ when $t\in[0,t_1)$.  Hence, by system \eqref{eq:ODE} we obtain
\begin{align*}
\left|\frac{\di}{\di t}x_i(t)\right|&\leq \frac{1}{2}\sum_{j<i}|p_i||p_j|+\frac{1}{2}\sum_{j>i}|p_i||p_j|+\sum_{1\leq m<i<n\leq N}|p_m||p_n|\\
&\leq \frac{1}{2}\left(\sum_{i=1}^N|p_i|\right)^2\leq \frac{1}{2}M_0^2
\end{align*}

\textbf{Step 2}. Proof of (iii). Obviously, we have
\[
u^N\in C([0,t_1);H^1(\mathbb{R}))\cap L^\infty(0,t_1; W^{1,\infty}(\mathbb{R}))\cap W^{1,\infty}(0,t_1;L^\infty(\mathbb{R})).
\]
In the following proof we denote $u=u^N$ and $m=m^N$. For any   test function $\varphi\in C_c^\infty(\mathbb{R}\times[0,t_1))$,  let
\begin{multline}\label{eq:weakfunctional2}
\mathcal{L}(u,\varphi)=\int_0^{t_1}\int_{\mathbb{R}}u(\varphi_t-\varphi_{txx})\di x\di t+\int_0^{t_1}\int_{\mathbb{R}}(u^3+2uu_x^2)\varphi_x-\frac{1}{3}u^3\varphi_{xxx}\di x\di t\\
-\int_0^{t_1}\int_{\mathbb{R}}\varphi_{x}\overline{u_x^2} m(\di x\di t)=:I_1+I_2+I_3.
\end{multline}
Denote $x_0:=-\infty$, $x_{N+1}:=+\infty$ and $p_0=p_{N+1}=0$. Notice that $m(x,t)=0$ for $x\in(x_i,x_{i+1})$ ($i=0,\cdots,N$) and $u_x(x_i-)-u_x(x_i+)=p_i$ ($i=1,\cdots,N$).
By integration by parts for space variable $x$, we calculate $I_1$ as
\begin{equation}\label{eq:I11}
\begin{aligned}
I_1&=\int_0^{t_1}\int_{\mathbb{R}}u(\varphi_t-\varphi_{txx})\di x\di t\\
&=\int_0^{t_1}\sum_{i=0}^{N}\int_{x_i}^{x_{i+1}}u(\varphi_t-\varphi_{txx})\di x\di t=\int_0^{t_1}\sum_{i=0}^{N}\int_{x_i}^{x_{i+1}}u\varphi_t+u_x\varphi_{tx}\di x\di t\\
&=\int_0^{t_1}\sum_{i=0}^{N}\int_{x_i}^{x_{i+1}}m\varphi_t\di x\di t+\int_0^{t_1}\sum_{i=0}^{N}\Big[(u_x\varphi_{t})(x_{i+1}-,t)-(u_x\varphi_{t})(x_{i}+,t)\Big]\di t\\
&=\int_0^{t_1}\sum_{i=1}^Np_i\varphi_t(x_i(t),t)\di t.
\end{aligned}
\end{equation}
Similarly for $I_2$, we have
\begin{equation}\label{eq:I21}
\begin{aligned}
I_2&=\int_0^{t_1}\int_{\mathbb{R}}(u^3+2uu_x^2)\varphi_x-\frac{1}{3}u^3\varphi_{xxx}\di x\di t=\int_0^{t_1}\int_{\mathbb{R}}(u^3+2uu_x^2)\varphi_x+u^2u_x\varphi_{xx}\di x\di t\\
&=\int_0^{t_1}\sum_{i=0}^N\int_{x_i}^{x_{i+1}}mu^2\varphi_x\di x\di t+\int_0^{t_1}\sum_{i=0}^N \Big[(u^2u_x\varphi_x)(x_{i+1}-)-(u^2u_x\varphi_x)(x_{i}+)\Big]\di t\\
&=\int_0^{t_1}\sum_{i=1}^N p_iu^2(x_i(t),t)\varphi_x(x_i(t),t)\di t.
\end{aligned}
\end{equation}
For $I_3$ in \eqref{eq:weakfunctional2}, we have
\begin{equation}\label{eq:I31}
I_3=-\int_0^{t_1}\sum_{i=1}^N p_i\overline{u_x^2}(x_i(t),t)\varphi_x(x_i(t),t)\di t.
\end{equation}
Combining  \eqref{eq:ODE}, \eqref{eq:weakfunctional2}, \eqref{eq:I11}, \eqref{eq:I21} and \eqref{eq:I31} gives
\begin{align}\label{useful}
\mathcal{L}(u,\varphi)=\sum_{i=1}^Np_i\int_0^{t_1}\frac{\di}{\di t}\varphi(x_i(t),t)\di t=-\int_{\mathbb{R}}\varphi(x,0) m_0^N(\di x).
\end{align}
Hence, $u^N$ defined by \eqref{eq:Npeakonsoluiton} is a weak solution in the sense of Definition \ref{def:weaksolution}.

\textbf{Step 3}.   Proof of (iv).  First, we have
\begin{equation}
\begin{aligned}
\mathcal{H}(u^N)=\int_{\mathbb{R}}(u^N)^2+(u_x^N)^2\di x&=\int_{\mathbb{R}}m^Nu^N\di x\\
&=\sum_{i=1}^Np_iu^N(x_i,t)=\sum_{i,j=1}^Np_ip_jG(x_i-x_j).
\end{aligned}
\end{equation}
Let $a_{ij}$ and $A^N_k$ be defined as in Proposition \ref{pro:energyconservation}. Then, we have $\frac{\di}{\di t}x_i(t)=A_i^N$, $i=1,\cdots, N$.  Therefore, for $i<j$,
\[
\frac{\di}{\di t}a_{ij}=a_{ij}(A_i^N-A_j^N).
\]
By \eqref{eq:energy} we obtain
\[
\frac{\di}{\di t}\mathcal{H}(u^N)=\frac{\di}{\di t}\sum_{1\leq i<j\leq N}2a_{ij}=2\sum_{1\leq i<j\leq N}a_{ij}(A_i^N-A_j^N)=0.
\]

\end{proof}

In Lemma \ref{lmm:Npeakonlemmano collision}, $u^N$ defined by \eqref{eq:Npeakonsoluiton} is a global weak solution when $t_1=+\infty$. However, when collision of $x_i(t)$ happens, we have $t_1<+\infty$. Next, we introduce the sticky particle method to extend the trajectories of $N$-peakon solutions globally.
The main idea for sticky particle method is very simple: whenever the trajectories collide with each other, we assume they stick together. Then, we are going to use the following four steps to extend the solution $x_i$ and $u^N$ when $t_1<+\infty$. Denote  $I:=\{1,\ldots,N\}$.

\begin{enumerate}
\item [\textbf{Sticky 1.}] Sticky amplitudes $q_k$ and  index $i_k$ after collision:\\
For $1\leq i\leq N$, denote the collection of the indices of $x_j$ coinciding with $x_i$ at time $t_1$
\begin{align}\label{collection}
J_i:=\big\{j:x_j(t_1)=x_i(t_1)\big\}.
\end{align}
We pick up the collection of minimal indices from $J_i$
\begin{align}\label{minimalindices}
\overline{I}:=\{\min J_i:i=1,\ldots,N\}=\{i_1<\ldots<i_{N_1}\}\subsetneqq I.
\end{align}
Define
\begin{align}\label{singlepeakon}
q_k:=\sum_{j\in  J_{i_k}}p_j,\quad 1\leq k\leq N_1.
\end{align}
Each set $J_{i_k}$ corresponds to a single peakon $q_k\delta(x-x_{i_k}(t_1))$.
$\{J_{i_k}\}_{k=1}^{N_1}$ is a partition of $I$, and we have
\begin{align}\label{summation}
\sum_{k=1}^{N_1}q_k=\sum_{i=1}^Np_i~\textrm{ and }~\sum_{k=1}^{N_1}|q_k|\leq \sum_{i=1}^N|p_i|\leq M_0.
\end{align}

\item [\textbf{Sticky 2.}]  New initial data $m_1$ and $y^0_k$:\\
Set
$$y_{k}^0:=x_{i_k}(t_1)~\textrm{ for }~1\leq k\leq N_1$$
and
$$m_1(x):=\sum_{k=1}^{N_1}q_k\delta(x-y^0_k).$$
By the definition of $q_k$  we know
\begin{align}\label{connectiontwotime}
m_1(x)=\sum_{k=1}^{N_1}q_k\delta(x-y^0_k)=\sum_{i=1}^Np_i\delta(x-x_i(t_1)).
\end{align}

\item [\textbf{Sticky 3.}]  $N_1$ peakon solution between two collision time $t_1$ and $t_2$:\\
Consider the system \eqref{eq:ODE} with $N_1$ particles $(1\leq k\leq N_1)$ (view $t_1$ as initial time)
\begin{equation}\label{eq:ODE2}
\left\{
\begin{aligned}
&\frac{\di}{\di t}y_k=\frac{1}{2}\sum_{j<k}q_kq_je^{y_j-y_k}+\frac{1}{2}\sum_{j>k}q_kq_je^{y_k-y_j}+\sum_{m<k<n}q_mq_ne^{y_m-y_n},\\
&y_{k}(t_1)=y_{k}^0.
\end{aligned}
\right.
\end{equation}
There exists a unique solution $\{y_k(t)\}_{k=1}^{N_1}$ to \eqref{eq:ODE2}. Because $x_{i_1}(t_1)<x_{i_2}(t_1)<\ldots<x_{i_{N_1}}(t_1)$,  we know $y_{1}^0<y_{2}^0<\ldots<y_{N_1}^0$.
Set
\[
t_2:=\sup\{t_0:y_1(t)<y_2(t)<\ldots<y_{N_1}(t) ~\textrm{ for }~t\in [t_1,t_0)\},
\]
and obviously $t_2>t_1$. Then similarly to \eqref{equilcontinuous},  we obtain from  \eqref{summation} that
\begin{align}\label{eq:equalocn}
\left|\frac{\di}{\di t}y_k(t)\right|\leq \frac{1}{2}M_0^2,\quad t\in [t_1,t_2), \quad i=1,\ldots,N
\end{align}
and
\begin{align}\label{defv}
v^{N_1}(x,t):=\sum_{k=1}^{N_1}q_kG(x-y_k(t)),\quad t\in[t_1,t_2)
\end{align}
is a weak solution to the mCH equation subject to initial data
$$(1-\partial_{xx})v^{N_1}(x,t_1)=m_1(x)=\sum_{i=1}^Np_i\delta(x-x_i(t_1)).$$

\item [\textbf{Sticky 4.}]  Extend in time for $x_i$ and $u^N$:\\
Extend $x_i$ in time by
\begin{align}\label{extendxi}
x_{i}(t)=y_k(t)~\textrm{ for }~i\in J_{i_k}~\textrm{ and }~t\in [t_1,t_2).
\end{align}
Because $\{J_{i_k}\}_{k=1}^{N_1}$ is a partition of $I$, all the trajectories $x_i$ ($1\leq i\leq N$) have been extended and $x_i(t)$ stick together as one trajectory $y_k(t)$ when $i\in J_{i_k}$.\\
Extend $u^N$ in time by
\begin{align}\label{extenduN}
u^N(x,t)=\sum_{i=1}^Np_iG(x-x_i(t))~\textrm{ for }~t\in [t_1,t_2).
\end{align}
Combining \eqref{summation}, \eqref{defv}, \eqref{extendxi} and \eqref{extenduN}, we know
\begin{align}\label{eq:equal}
u^N(x,t)=v^{N_1}(x,t)~\textrm{ for }~t\in [t_1,t_2).
\end{align}
Hence, $u^N(x,t)$ is a weak solution to the mCH equation  when $t\in [t_1,t_2)$.

\end{enumerate}

Next, we prove that $u^N$ defined by \eqref{eq:Npeakonsoluiton} and \eqref{extenduN}  is a weak solution in $[0,t_2).$ For any test function $\varphi\in C_c^\infty(\mathbb{R}\times[0,t_2))$,   by using \eqref{useful} and \eqref{connectiontwotime} we have
\begin{align*}
\mathcal{L}(u^N,\varphi)&=\sum_{i=1}^Np_i\int_0^{t_1}\frac{\di}{\di t}\varphi(x_i(t),t)\di t+\sum_{k=1}^{N_1}q_k\int_{t_1}^{t_2}\frac{\di}{\di t}\varphi(y_{k}(t),t)\di t\\
&=\sum_{i=1}^Np_i\varphi(x_i(t_1),t_1)-\sum_{i=1}^Np_i\varphi(c_i,0)-\sum_{k=1}^{N_1}q_k\varphi(x_{i_k}(t_1),t_1)\\
&=-\sum_{i=1}^Np_i\varphi(c_i,0)=-\int_{\mathbb{R}}\varphi(x,0)m_0^N(\di x),
\end{align*}
which means $u^N(x,t)$ is a weak solution to the mCH equation subjcet to initial data $m^N_0$ (given by \eqref{eq:m0N}) when $t\in [0,t_2)$.
Moreover, by \eqref{equilcontinuous} and \eqref{eq:equalocn} we know
\[
\left|\frac{\di}{\di t}x_i(t)\right|\leq \frac{1}{2}M_0^2,\quad t\in[0,t_2),\quad i=1,\ldots,N.
\]

If $t_2=+\infty$, then we obtain a global weak solution to the mCH equation. If $t_2<+\infty$, then we can repeat the above process to extend trajcetories $x_i$ and weak solution $u^N$ in time. By the sticky assumption, collision between $x_i$ can  happen at most $N-1$ times and we can extend $x_i$ and $u^N$ globally.  Moreover, at each time interval $x_i$ is unique. Hence, the sticky weak solution constructed by the above method is unique.
We have the following theorem:
\begin{theorem}\label{thm:sticky}
Let $m^N_0$ given by  \eqref{eq:m0N}. Then, there are unique global sticky trajectories $\{x_i(t)\}_{i=1}^N$ (defined similarly by \eqref{eq:ODE2} and \eqref{extendxi} at each collision time interval) with  initial data $\{x_i(0)=c_i\}_{i=1}^N$ and the following estimate holds
\begin{align}\label{eq:Lipschitz}
\left|\frac{\di}{\di t}x_i(t)\right|\leq \frac{1}{2}M_0^2,\quad t\geq0,\quad i=1,\ldots, N.
\end{align}
There is a  global sticky peakon weak solution to the mCH equation with initial data $m^N_0$ and it is given by
\begin{align}\label{approximationsolution}
u^N(x,t)=\sum_{i=1}^Np_iG(x-x_i(t))~\textrm{ and }~m^N(x,t)=\sum_{i=1}^Np_i\delta(x-x_i(t)).
\end{align}
Moreover, the energy $\mathcal{H}(u^N)$ is conserved.

\end{theorem}
\begin{proof}
We use the notations in steps Sticky 1 to Sticky 4. We only need to show the continuity of $\mathcal{H}(u^N)$ at the first collision time $t_1$. For $t\in[0,t_1)$, by Lemma \ref{lmm:Npeakonlemmano collision} we have
\begin{align*}
\mathcal{H}(u^N)(t)=\sum_{i,j=1}^Np_ip_jG(x_i(t)-x_j(t))\equiv \mathcal{H}(u^N)(0).
\end{align*}
Since $x_i(t)$, $i=1,\cdots,N$, are continuous, we have
\begin{multline*}
\mathcal{H}(u^N)(0)=\mathcal{H}(u^N)(t_1-)=\sum_{i,j=1}^Np_ip_jG(x_i(t_1)-x_j(t_1))\\
=\sum_{i=1}^N\sum_{\ell=1}^{N_1}p_iq_\ell G(x_i(t_1)-y_\ell^0)=\sum_{k,\ell=1}^{N_1}q_kq_\ell G(y_k^0-y_{\ell}^0)=\mathcal{H}(v^{N_1})(t_1).
\end{multline*}
By the construction of sticky $N$-peakon, we have
\begin{align*}
\mathcal{H}(u^N)(t)=\mathcal{H}(v^{N_1})(t),\quad t\in[t_1,t_2).
\end{align*}
Using Lemma \ref{lmm:Npeakonlemmano collision} again, we know
\[
\mathcal{H}(v^{N_1})(t)\equiv \mathcal{H}(v^{N_1})(t_1),\quad t\in[t_1,t_2),
\]
and hence
\begin{align*}
\mathcal{H}(u^N)(t)=\mathcal{H}(v^{N_1})(t)=\mathcal{H}(v^{N_1})(t_1)=\mathcal{H}(u^N)(0),\quad t\in[t_1,t_2),
\end{align*}
which means
\[
\mathcal{H}(u^N)(t)\equiv\mathcal{H}(u^N)(0),\quad t\in[0,t_2).
\]

\end{proof}

\section{A selection principle for the uniqueness}\label{sec:dispersion}
In this section, we will provide a dispersive regularization as a selection principle for the uniqueness of the conservative $N$-peakon solutions. Let $\{x_i(t)\}_{i=1}^N$ be the global sticky trajectories constructed in Section \ref{sec:sticky}, and $u^N(x,t)=\sum_{i=1}^Np_iG(x-x_i(t))$. As mentioned in Introduction, when the peakons are well separated, i.e., $x_1(t)<\cdots<x_N(t)$, the vector field of \eqref{eq:ODE} is the same as $[(u^N)^2-\overline{(u_x^N)^2}](x_i(t),t)$. 
Consider compactly supported even mollifiers $\rho_\e\geq0$ ($\e>0$) with compact support in $(-\e,\e)$, $\|\rho_\e\|_{L^1}=1$, and increasing in $(-\e,0)$; see, e.g., \cite[Chapter 4]{Brezis}.
For a piecewise continuous function $f:\mathbb{R}\to\mathbb{R}$ with a jump discontinuity at $x_0$, it holds 
\begin{align}\label{eq:fact}
\overline{f}(x_0)=\lim_{\e\to0}(\rho_\e\ast f)(x_0).
\end{align}
See Appendix \ref{fact} for a proof of \eqref{eq:fact}.
This motivates the following regularization method for the $N$-peakon solutions.
Fix an integer $N>0$. Give an initial datum as \eqref{eq:m0N} for some constants $p_i$, $c_i$ ($1\leq i\leq N$) and  $M_0$. For any $N$ particles $\{x_k\}_{k=1}^N\subset\mathbb{R}$, define 
\begin{align*}
u^{N}(x;\{x_k\}_{k=1}^N):=\sum_{k=1}^Np_kG(x-x_k),\quad U^N(x;\{x_k\}_{k=1}^N):=\left[(u^{N})^2-(\partial_xu^{N})^2\right](x;\{x_k\}_{k=1}^N),
\end{align*}
and
\begin{align*}
U^{N,\e}(x;\{x_k\}_{k=1}^N):=(\rho_\e\ast U^N)(x;\{x_k\}_{k=1}^N).
\end{align*}
The system of ODEs for dispersive regularization is given by
\begin{align}\label{approximateODE}
\frac{\di}{\di t}x_i^\e(t)=U^{N,\e}(x^\e_i(t);\{x_k^\e(t)\}_{k=1}^N),\quad i=1,\cdots,N,
\end{align}
with initial data $x_i^\e(0)=c_i$ given in \eqref{eq:m0N}.  
Because $U^{N,\e}$ is  Lipschitz continuous and bounded, the existence and uniqueness of a global solution $\{x_i^\e(t)\}_{i=1}^N$ to this system of ODEs follow from  standard ODE theories. By using the solution $\{x_i^\e(t)\}_{i=1}^N$, we set
\begin{align*}
u^{N,\e}(x,t):=u^{N,\e}(x;\{x_k^\e(t)\}_{k=1}^N).
\end{align*}
Set
\begin{align}\label{eq:UNepisilon}
U_\e^N(x,t):=U^N(x;\{x_k^\e(t)\}_{k=1}^N),\quad	U^{N,\e}(x,t):=U^{N,\e}(x;\{x_k^\e(t)\}_{k=1}^N).
\end{align}
Therefore,  $U^{N,\e}(x,t)=(\rho_\e\ast U^N_\e)(x,t)$ and \eqref{approximateODE}  can be rewritten as
\begin{align}\label{eq:approximateODE}
\frac{\di }{\di t}x_i^\e(t)=U^{N,\e}(x^\e_i(t),t),\quad i=1,\cdots,N.
\end{align}

\begin{remark}
In \cite{gao2018dispersive}, a double mollification method was introduced for system \eqref{eq:ODE1}, where the mollified vector field was given by
\[
\tilde{U}^{N,\e}(x;\{x_k\}_{k=1}^N)=\rho_\e\ast\left[(\rho_\e\ast u^N)^2-(\rho_\e\ast u_x^N)^2\right].
\]
Since $G$ is continuous, we have $\lim_{\e\to0}\rho_\e\ast[(\rho_\e\ast G)^2](0)=G^2(0)=\frac{1}{4}$. Moreover, for this double mollification,  one can prove  $\lim_{\e\to0}\rho_\e\ast[(\rho_\e\ast G_x)^2](0)=\frac{1}{12}$; see \cite[Lemma 2.1]{gao2018dispersive}. Therefore
\[
\lim_{\e\to0}\rho_\e\ast[(\rho_\e\ast G)^2-(\rho_\e\ast G_x)^2](0)=\frac{1}{4}-\frac{1}{12}=\frac{1}{6},
\]
which is compatible with the traveling speed of one peakon solution given by \eqref{eq:ODE1}. Both the sticky particle methods in \cite{GaoLiu} and the limit of the above double mollification method will generate global peakon solutions corresponding to system \eqref{eq:ODE1}. However these two kind of solutions may be different; see \cite[Section 3.3]{gao2018dispersive}.

For the one time mollification in \eqref{eq:approximateODE}, one can show that
\[
\lim_{\e\to0}\rho_\e\ast(G^2-G_x^2)(0)=0.
\]
which is compatible with the traveling speed of one peakon solution given by \eqref{eq:ODE}. Simulation results show that the approximate trajectories given by \eqref{eq:approximateODE} converge to  the sticky trajectories in Theorem \ref{thm:sticky} (see Figure \ref{fig:dispersivelimit} and Figure \ref{fig:Four} below), which is different with the double mollification method for system \eqref{eq:ODE1}. Indeed. In Theorem \ref{thm:dispersionlimit}, we will prove the convergence rigorously before collision time. We will also give a rigorous proof in Proposition \ref{pro:dispersionlimit} for three peakons ($N=3$) that the limit of the trajectories given by \eqref{eq:approximateODE} is exactly the sticky trajectories even after the collision time. 

\end{remark}
\begin{figure}[H]
\begin{center}
\includegraphics[width=0.7\textwidth]{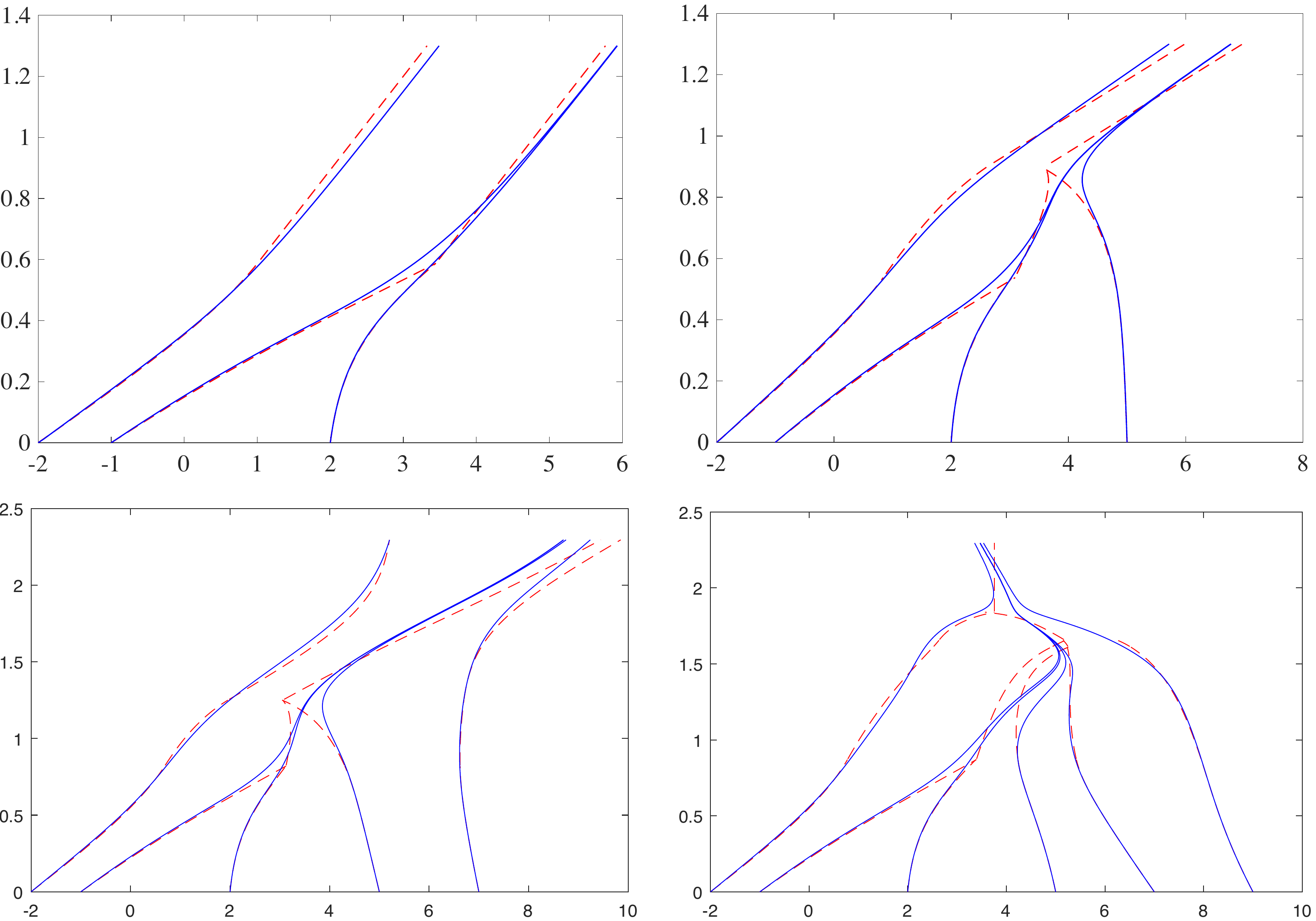}
\end{center}
\caption{The blue solid lines are approximated trajectories given by \eqref{approximateODE}, while the red dashed lines are the sticky trajectories.}
\label{fig:dispersivelimit}
\end{figure}
\begin{figure}[H]
\begin{center}
\includegraphics[width=0.7\textwidth]{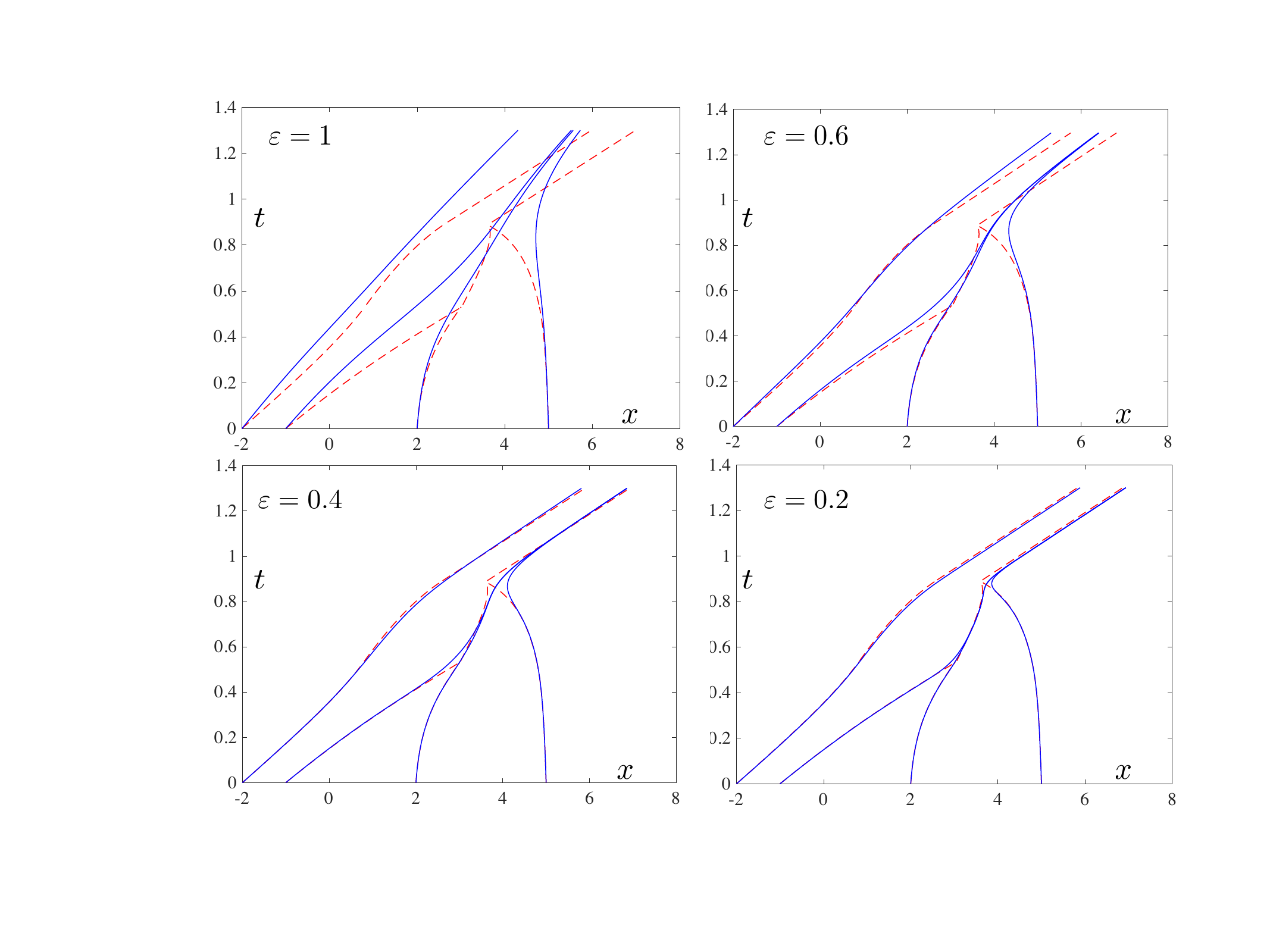}
\end{center}
\caption{Evolution of trajectories given by \eqref{approximateODE} as $\e\to0$ (blue solid lines), and comparison with sticky particle trajectories (red dashed lines). Momentum: $p_1=15,~p_2=2,~p_3=3,~p_4=-2$; initial positions: $x_1(0)=-2,~x_2(0)=-1,~x_3(0)=-2,~x_4(0)=5$.
The blue solid lines are regularized trajectories of four peakons $\{x^{\e}_i(t)\}_{i=1}^4$ given by \eqref{approximateODE}. The red dashed lines are trajectories of sticky four peakons.}
\label{fig:Four}
\end{figure}
Note that the regularized system \eqref{eq:approximateODE} for the $N$-peakon solutions can be equivalently reformulated as the regularization performed directly on the mCH equation \eqref{eq:mCH}:
\begin{align}\label{eq:m}
m_t+[\rho_\e\ast(u^2-u_x^2)m]_x=0,\quad m=u-u_{xx}.
\end{align}
Heuristically, we consider for smooth functions here. Since $\rho_\e$ is an even function, we have  the expansion 
\[
\rho_\e\ast f=\int_{\mathbb{R}}\rho_\e(x)f(x-y)\di y=\int_{\mathbb{R}}\rho_\e(x)\left[f(x)-yf'(x)+\frac{1}{2}y^2f_{xx}(x)+\cdots\right]\di y= f+\kappa\e^2 f_{xx}+O(\e^4)
\] 
with $\kappa=\int_{\mathbb{R}}y^2f(y)\di y/2$.
We do linearization of \eqref{eq:m} around the constant solutions.  Let $u = 1 + \delta v$ and $m = u -u_{xx} = 1 + \delta n$, where $n = v-v_{xx}$ and $\delta\in\mathbb{R}$. Keeping orders up to $\e^2$ and $\delta$, the following linearized
equation is obtained:
\[
n_t+(n+2v)_x+2\kappa\e^2 v_{xxx}+O(\delta)+O(\e^4)=0.
\]
The leading term corresponding to the mollification is a dispersive term $2\kappa\e^2v_{xxx}$. Hence, the regularized system \eqref{eq:approximateODE} has some dispersive effects. 

Next, we  show the collision avoidance of the dispersive regularization. 
\begin{proposition}\label{pro:nocollisionepsilon}
Let $\{p_i\}_{i=1}^N$ and $\{c_i\}_{i=1}^N$ be the same as in \eqref{eq:m0N}.  Let $\{x_i^\e(t)\}_{i=1}^N$ be the approximate trajectories given by \eqref{eq:approximateODE} subject to $x_i^\e(0)=c_i$, $i=1,\ldots, N$. Then, the approximate trajectories never collide with each other, i.e., $x^\e_1(t)<x^\e_2(t)<\cdots<x_N^\e(t)$ for all $t>0$. 
\end{proposition}
\begin{proof}
If collisions between $\{x_i^\e\}_{i=1}^N$ happen in finite time, we assume that the first collision is between  $x_k^\e$ and $x_{k+1}^\e$ for some $1\leq k\leq N-1$ at time $0<T_*<+\infty$. We have
\[
\frac{\di}{\di t}x^\e_k(t)=U^{N,\e}(x^\e_k(t),t)=\int_{\mathbb{R}}\rho_\e(x^\e_k(t)-x)U_\e^N(x,t)\di x,
\]
and
\[
\frac{\di}{\di t}x^\e_{k+1}(t)=(\rho_\e\ast U_\e^{N})(x^\e_{k+1}(t),t)=\int_{\mathbb{R}}\rho_\e(x^\e_{k+1}(t)-x)U^N_\e(x,t)\di x.
\]
where $U_\e^N$ is defined by \eqref{eq:UNepisilon}. For $t< T_*$, taking the difference gives
\[
\frac{\di}{\di t}(x_{k+1}^\e-x^\e_k)=\int_{\mathbb{R}}\Big[\rho_\e(x^\e_{k+1}(t)-x)-\rho_\e(x^\e_{k}(t)-x)\Big]U^N_\e(x,t)\di x.
\]
Because $\|U_\e^N\|_{L^\infty}\leq \frac{M_0^2}{2}$, we have $|x_i^\e(t)|\leq |c_i|+\frac{M_0^2}{2}T^*$ and
\begin{align*}
\left|\frac{\di}{\di t}(x_{k+1}^\e-x^\e_k)\right|\leq&\frac{M_0^2}{2}\int_{x_k^\e(t)-\e}^{x_{k+1}^\e(t)+\e}\left|\rho_\e(x^\e_{k+1}(t)-x)-\rho_\e(x^\e_{k}(t)-x)\right|\di x\\
\leq &\frac{M_0^2}{2}(2\e+x_{k+1}^\e-x_k^\e)\|\partial_x\rho_\e\|_{L^\infty}(x_{k+1}^\e-x_k^\e)\leq C(x_{k+1}^\e-x_k^\e),~~t< T_*,
\end{align*}
where $C=C(\e,c_i,T^*,M_0)$ is a constant depending on $\e,~c_i,~T^*,~M_0$.
Hence, for $t<T_*$ we have
\[
-C(x_{k+1}^\e-x_k^\e)\leq \frac{\di}{\di t}(x_{k+1}^\e-x^\e_{k})\leq  C(x_{k+1}^\e-x_k^\e),
\]
which implies
\[
0<(c_{k+1}-c_k)e^{-C t}\leq x_{k+1}^\e(t)-x_k^\e(t)~\textrm{ for }~t< T_*.
\]
This is a contradiction with the assumption for $T_*$. Hence, $x^\e_1(t)<x^\e_2(t)<\cdots<x_N^\e(t)$ for all $t>0$.

\end{proof}

The following theorem shows that before collision, the dispersion limit gives sticky peakons.
\begin{theorem}\label{thm:dispersionlimit}
Let $\{p_i\}_{i=1}^N$ and $\{c_i\}_{i=1}^N$ be the same as in \eqref{eq:m0N}. Let $\{x_i(t)\}_{i=1}^N$ be the sticky trajectories constructed in Section \ref{sec:sticky} with initial data $x_i(0)=c_i$, and $\{x^\e_i(t)\}_{i=1}^N$ be the approximated trajectories given by \eqref{approximateODE} with $x_i^\e(0)=c_i$. Then the approximated trajectories converge to the sticky trajectories at least up to the first collision time. More precisely, if the first collision between sticky trajectories $\{x_i(t)\}_{i=1}^N$ happens at $t_1>0$, then we have
\begin{align}\label{eq:limit}
\lim_{\e\to0}x_k^\e(t)=x_k(t),\quad k=1,\cdots,N,~~t\in[0,t_1].
\end{align}
\end{theorem}

\begin{proof}
 Let $u^N(x,t)=\sum_{i=1}^Np_iG(x-x_i(t))$. From the construction of the sticky trajectories $\{x_i(t)\}_{i=1}^N$, the following ODE holds for any $0<t<t_1$:
\begin{align}\label{eq:sticky}
\frac{\di}{\di t}x_k(t)=\left[(u^N)^2-\overline{(u_x^N)^2}\right](x_k(t),t),\quad k=1,2,\cdots,N.
\end{align}
For  $t<t_1$ and $x_k(t)<x<x_{k+1}(t)$, we have
\begin{align}\label{eq:Ak+}
\left[(u^N)^2-(u_x^N)^2\right](x,t)=\sum_{j>k}p_jp_ke^{x_k(t)-x_j(t)}+\sum_{i<k<j}p_ip_je^{x_i(t)-x_j(t)}=:A_k^+(t).
\end{align}
For  $t<t_1$ and $x_{k-1}(t)<x<x_{k}(t)$,  we have
\begin{align}\label{eq:Ak-}
\left[(u^N)^2-(u_x^N)^2\right](x,t)=\sum_{i<k}p_ip_ke^{x_i(t)-x_k(t)}+\sum_{i<k<j}p_ip_je^{x_i(t)-x_j(t)}=:A_{k}^-(t).
\end{align}
From the definition, we have $A_k^+(t)=A_{k+1}^-(t)$, $A_1^-=A_N^+=0$, and
\begin{align}\label{eq:ode1}
\frac{\di }{\di t}x_k(t)=\frac{A_k^+(t)+A_k^-(t)}{2},\quad t<t_1, k=1,\cdots,N.
\end{align}
Due to collision avoidance of $x_i^\e(t)$, we define similarly
\begin{equation}
\begin{aligned}\label{eq:AE}
A_{\e k}^+(t):=&U_\e^N(x,t)=\sum_{j>k}p_jp_ke^{x^\e_k(t)-x^\e_j(t)}+\sum_{i<k<j}p_ip_je^{x^\e_i(t)-x^\e_j(t)},~~x_k^\e(t)<x<x_{k+1}^\e(t), \\
A_{\e k}^-(t):=&U_\e^N(x,t)=\sum_{i<k}p_ip_ke^{x^\e_i(t)-x^\e_k(t)}+\sum_{i<k<j}p_ip_je^{x^\e_i(t)-x^\e_j(t)},~~x_{k-1}^\e(t)<x<x_{k}^\e(t).
\end{aligned}
\end{equation}

Note that for any $\e>0$, we have
$
|U^{N,\e}|\leq \frac{M_0^2}{2},
$
where $U^{N,\e}$ was defined by \eqref{eq:UNepisilon}. Let
\begin{align*}
t^*:=\frac{2\min\{c_{k+1}-c_k,~~1\leq k\leq N-1\}}{M_0^2}.
\end{align*}
Then for any $\e>0$, we have
\begin{align*}
x_{k+1}^\e(t)-x_k^\e(t)>c_{k+1}-c_k-\frac{M_0^2}{2}t, \quad t\in[0,t^*).
\end{align*}
This implies that for any $\tilde{t}\in[0,t^*)$, there exists $\e_0$ small enough such that
\begin{align}\label{eq:uniforme}
\min_{1\leq k\geq N-1,t\in[0,\tilde{t}]}(x_{k+1}^\e(t)-x_{k}^\e(t))>\e_0.
\end{align}
Then for $0<\e<\e_0$ we have
\begin{align*}
\frac{\di }{\di t}x^\e_k(t)=\int_{\mathbb{R}}\rho_\e(x_k^\e(t)-y)U^N_\e(y,t)\di y=\int_{x_k^\e(t)-\e}^{x_k^\e(t)+\e}\rho_\e(x_k^\e(t)-y)U^N_\e(y,t)\di y.
\end{align*}
Combining \eqref{eq:AE} and \eqref{eq:uniforme} gives
\begin{align*}
\frac{\di }{\di t}x^\e_k(t)=&\int_{x_k^\e(t)-\e}^{x_k^\e(t)}\rho_\e(x_k^\e(t)-y)A_{\e k}^-(t)\di y+\int^{x_k^\e(t)+\e}_{x_k^\e(t)}\rho_\e(x_k^\e(t)-y)A_{\e k}^+(t)\di y\\
&=A_{\e k}^-(t)\int_{-\e}^0\rho_\e(y)\di y+A_{\e k}^+(t)\int^{\e}_0\rho_\e(y)\di y\\
&=\frac{A_{\e k}^-(t)+A_{\e k}^+(t)}{2},\quad k=1,\cdots, N.
\end{align*}
The above ODE system is the same as \eqref{eq:ode1}. Therefore, for $t\in[0,\tilde{t}]$ and $\e<\e_0$, we have $x_i^\e(t)=x_i(t)$. Because $\tilde{t}$ is arbitrarily chosen in $[0,t^*)$, we obtain
\begin{align*}
\lim_{\e\to0}x_k^\e(t)=x_k(t),\quad k=1,\cdots,N,~~t\in[0,t^*).
\end{align*}
If $t^*<t_1$, then for $\e_0>0$ small enough, we have 
\[
\min_{1\leq k\leq N-1,~t\in[0,t^*]}(x_{k+1}^\e(t)-x_{k}^\e(t))>\e_0
\]
for any $0<\e<\e_0$. Then the above proof shows that 
\[
x_k^\e(t)=x_k(t),\quad k=1,\cdots,N,~~ 0<\e<\e_0,~~t\in[0,t^*].
\]
Next, view $\{x_k(t^*)\}_{k=1}^N$ as new initial data and proceed the above process again. We can prove that for any $\tau\in[0,t_1)$, there exists $\e_\tau >0$ small enough such that
$x_k^\e(t)=x_k(t)$ for any $k=1,\cdots, N$, $t\in[0,\tau]$, and $0<\e<\e_\tau$. Therefore, we have
\begin{align}\label{eq:limit1}
\lim_{\e\to0}x_k^\e(t)=x_k(t),\quad k=1,\cdots,N,~~t\in[0,t_1).
\end{align}
For the convergence at $t=t_1$, let $\Delta t>0$ and we have 
\[
|x_k^\e(t_1)-x_k(t_1)|\leq |x_k^\e(t_1)-x_k^\e(t_1-\Delta t)|+|x_k^\e(t_1-\Delta t)-x_k(t_1-\Delta t)|+|x_k(t_1-\Delta t)-x_k(t_1)|.
\]
Since all the trajectories (approximated ones or sticky ones) are Lipschitz continuous with Lipschitz constant bounded by $M_0^2/2$, we can choose $\Delta t$ small enough such that the first and the third terms are small. Then we use \eqref{eq:limit1} to choose $\e$ small enough to get the smallness of the second term. This will finish the proof of \eqref{eq:limit}.

\end{proof}

Next, we prove the limit after collision rigorously for $N=3$. We have
\begin{proposition}\label{pro:dispersionlimit}
Let $\{p_i\}_{i=1}^3$ and $\{c_i\}_{i=1}^3$ be the same as in \eqref{eq:m0N} with $N=3$,  $\{x_i(t)\}_{i=1}^3$ be the sticky trajectories constructed in Section \ref{sec:sticky} with  $x_i(0)=c_i$, and $\{x^\e_i(t)\}_{i=1}^3$ be the approximated trajectories given by \eqref{approximateODE} with $x_i^\e(0)=c_i$. Then for any $t>0$, we have $x_i^\e(t)\to x_i(t)$ as $\e\to0$.
\end{proposition}
\begin{proof}
If the collision between the trajectories of three peakons happens at $t_1>0$, there will  be three possible cases:

\noindent Case 1:
\[
\left\{\begin{split}
&x_1(t)<x_2(t)<x_3(t),~~0\leq t<t_1;\\
&x_1(t)<x_2(t)=x_3(t),~~t\geq t_1.
\end{split}
\right.
\]

\noindent Case 2:
\[
\left\{\begin{split}
&x_1(t)<x_2(t)<x_3(t),~~0\leq t<t_1;\\
&x_1(t)=x_2(t)<x_3(t),~~t\geq t_1.
\end{split}
\right.
\]

\noindent Case 3:
\[
\left\{\begin{split}
&x_1(t)<x_2(t)<x_3(t),~~0\leq t<t_1;\\
&x_1(t)=x_2(t)=x_3(t),~~t\geq t_1.
\end{split}
\right.
\]
We will only prove Case 1 and Case 3. The proof of the second case is similar to Case 1.

\textbf{Proof of Case 1:}
First, we give some properties for $p_i$ under the assumption of Case 1. For $t<t_1$, we have
\begin{align}\label{eq:deriva12}
\frac{\di}{\di t}x_1(t)=\frac{1}{2}A_1^+(t)=\frac{1}{2}A_2^-(t),\quad\frac{\di}{\di t}x_2(t)=\frac{A_2^-(t)+A_2^+(t)}{2},
\end{align}
and
\begin{align}\label{eq:deriva3}
\frac{\di}{\di t}x_3(t)=\frac{A_3^-(t)+A_3^+(t)}{2}=\frac{1}{2}A_3^-(t)=\frac{1}{2}A_2^+(t).
\end{align}
Here $A_k^+$ and $A_k^-$ are defined by \eqref{eq:Ak+} and \eqref{eq:Ak-} for $N=3$.
Therefore, 
\begin{equation}\label{eq:prior}
\begin{aligned}
\frac{\di}{\di t}(x_3(t)-x_2(t))=&-\frac{1}{2}A_2^-(t)=-\frac{1}{2}p_1p_2e^{x_1(t)-x_2(t)}-\frac{1}{2}p_1p_3e^{x_1(t)-x_3(t)}\\
=&-\frac{1}{2}p_1e^{x_1(t)-x_2(t)}(p_2+p_3e^{x_2(t)-x_3(t)}).
\end{aligned}
\end{equation}
Without loss of generality, assume $p_1>0$. We claim that
\begin{align}\label{eq:claim}
p_2+p_3e^{-a}>0~\textrm{ for }~0\leq a\leq c_3-c_2=x_3(0)-x_2(0).
\end{align}
\begin{proof}[Proof of \eqref{eq:claim}]
We first prove the claim for $a>0$. We only need to prove for $a=c_3-c_2$. Seeking for a contradiction, we assume $p_2+p_3e^{-(c_3-c_2)}\leq 0.$ If $p_2+p_3e^{-(c_3-c_2)}= 0$, then $x_3(t)-x_2(t)\equiv c_3-c_3>0$ for $t\geq0$ is the unique solution of \eqref{eq:prior}, which is  a contradiction with the assumption that $x_3(t_1)-x_2(t_1)=0.$ If $p_2+p_3e^{-(c_3-c_2)}<0$, since $p_1>0$, from \eqref{eq:prior} we have
\[
\frac{\di}{\di t}(x_3(t)-x_2(t))\Big|_{t=0}=-\frac{1}{2}p_1e^{x_1(0)-x_2(0)}(p_2+p_3e^{-(c_3-c_2)})>0,
\]
which implies $x_3(t)-x_2(t)>x_3(0)-x_2(0)=c_3-c_2$ for small time $t>0$. Notice that $\sup_{0\leq t\leq t_1}|x_1(t)-x_2(t)|=b>0$ is bounded, by continuity of the right hand side  of \eqref{eq:prior}, the derivative of $x_3(t)-x_2(t)$ will be positive when $x_3(t)-x_2(t)$ is near $c_3-c_2$, which makes  $x_3(t)-x_2(t)$ increasing. Hence one can prove $x_3(t)-x_2(t)\geq c_3-c_2$ for $0\leq t\leq t_1$, which
gives a contradiction with the assumption that $x_3(t_1)-x_2(t_1)=0.$ 

For the case $a=0$,  we will prove the claim for two cases: $p_3>0$ or  $p_3<0$.   If $p_3>0$, then 
\[
p_2+p_3> p_2+p_3e^{c_2-c_3}>0.
\]
If $p_3<0$, then $p_2>0$. Otherwise from \eqref{eq:prior} we know that $x_3(t)-x_2(t)$ is nondecreasing, which is a contradiction. Due to claim \eqref{eq:claim} for the case $a>0$, we know that
\[
p_2+p_3\geq 0.
\]
In this case, if $p_2+p_3=0$, we have
\[
\frac{\di}{\di t}(x_3(t)-x_2(t))=-\frac{1}{2}p_1p_2e^{x_1(t)-x_2(t)}(1-e^{x_2(t)-x_3(t)})\geq -\frac{1}{2}p_1p_2(1-e^{x_2(t)-x_3(t)})
\]
which implies
\[
\frac{\di}{\di t}\ln[e^{x_3(t)-x_2(t)}-1]\geq -\frac{1}{2}p_1p_2.
\]
Hence,
\[
x_3(t)-x_2(t)\geq \ln \left[1+(e^{c_3-c_2}-1)e^{-\frac{1}{2}p_1p_2t}\right]>0, \quad 0<t<t_1.
\]
and it is a contradiction with the assumption that $x_3(t_1)-x_2(t_1)=0.$ This proves claim \eqref{eq:claim} for $a=0$.
\end{proof}

Next, we analyze the vector field of sticky peakons for $t\geq t_1$. The sticky peakon solution is
\[
u(x,t)=p_1G(x-x_1(t))+(p_2+p_3)G(x-x_2(t)), ~~t\geq t_1.
\]
When $x_1(t)<x<x_2(t)$, we have
\begin{align*}
\left(u^2-\overline{u^2_x}\right)(x,t) =& \left[\frac{1}{2}p_1 e^{x_1-x}+\frac{1}{2}(p_2+p_3) e^{x-x_2}\right]^2-\left[-\frac{1}{2}p_1 e^{x_1-x}+\frac{1}{2}(p_2+p_3) e^{x-x_2}\right]^2\\
=&p_1(p_2+p_3)e^{x_1-x_2}.
\end{align*}
When $x<x_1(t)$ or $x>x_2(t)$, we have $(u^2-\overline{u^2_x})(x,t)=0.$ Therefore, we have
\begin{align}\label{eq:123}
\frac{\di}{\di t} x_i(t) = \frac{1}{2}p_1(p_2+p_3)e^{x_1-x_2}=\frac{1}{2}p_1p_2e^{x_1-x_2}+\frac{1}{2}p_1p_3e^{x_1-x_3}=\frac{1}{2}A_1^{+}(t)=\frac{1}{2}A_2^{-}(t)
\end{align}
for $t\geq t_1,~~i=1,2,3.$

For the approximated trajectories given by system \eqref{eq:approximateODE}, we have
\[
\frac{\di }{\di t}x_1^\e(t)=\int_{\mathbb{R}}\rho_\e(x_1^\e(t)-y)U_\e^N(y,t)\di y=\int_{x_1^\e(t)-\e}^{x_1^\e(t)+\e}\rho_\e(x_1^\e(t)-y)U_\e^N(y,t)\di y
\]
For $\e$ small enough, we have $x_1^\e(t_1)+\e <x_2^\e(t_1)$. Therefore, we have
\[
U_\e^N(y,t)=[(u^\e)^2-(u_x^\e)^2](y,t)=A_{\e1}^{-}(t),~~x_1^\e(t)-\e<y<x_1^\e(t),
\]
and
\[
U_\e^N(y,t)=[(u^\e)^2-(u_x^\e)^2](y,t)=A_{\e1}^{+}(t)=p_1p_2e^{x^\e_1-x^\e_2}+p_1p_3e^{x^\e_1-x^\e_3},~~x_1^\e(t)<y<x_1^\e(t)+\e.
\]
Using $A_{\e1}^{-}(t)=0$, we have
\begin{align}\label{eq:1e}
\frac{\di }{\di t}x_1^\e(t)=\int_{x_1^\e(t)}^{x_1^\e(t)+\e}\rho_\e(x_1^\e(t)-y)\di y A_{\e1}^{+}(t)=\frac{1}{2}A_{\e1}^{+}(t),~~t\geq t_1.
\end{align}
Similarly, 
\begin{equation}\label{eq:2e}
\begin{aligned}
\frac{\di}{\di t}x_2^\e(t)=&\int_{x_2^\e(t)-\e}^{x_2^\e(t)+\e}\rho_\e(x_2^\e(t)-y)U_\e^N(y,t)\di y\\
=&\int_{x_2^\e(t)-\e}^{x_2^\e(t)}\rho_\e(x_2^\e(t)-y)\di y\cdot A_{\e2}^{-}(t)+\int_{x_2^\e(t)}^{x_3^\e(t)}\rho_\e(x_2^\e(t)-y)\di y\cdot A_{\e2}^{+}(t)\\
=&\frac{1}{2}A_{\e2}^{-}(t)+\int_{0}^{x_3^\e(t)-x_2^\e(t)}\rho_\e(y)\di y\cdot A_{\e2}^{+}(t),
\end{aligned}
\end{equation}
and
\begin{equation}\label{eq:3e}
\begin{aligned}
\frac{\di}{\di t}x_3^\e(t)=&\int_{x_3^\e(t)-\e}^{x_3^\e(t)+\e}\rho_\e(x_3^\e(t)-y)U_\e^N(y,t)\di y
=\int_{x_3^\e(t)-\e}^{x_3^\e(t)}\rho_\e(x_3^\e(t)-y)U_\e^N(y,t)\di y\\
=&\int_{-\e}^{x_2^\e(t)-x_3^\e(t)}\rho_\e(y)\di y\cdot A_{\e2}^{-}(t)+\int^0_{x_2^\e(t)-x_3^\e(t)}\rho_\e(y)\di y\cdot A_{\e2}^{+}(t).
\end{aligned}
\end{equation}
Here, we used the fact that $A_{\e3}^+(t)=0$. Since $\rho_\e$ is even, taking the difference of \eqref{eq:2e} and \eqref{eq:3e} yields
\begin{equation}\label{eq:differ}
\begin{aligned}
\frac{\di}{\di t}(x_3^\e(t)-x_2^\e(t))&=-\int^0_{x_2^\e(t)-x_3^\e(t)}\rho_\e(y)\di y\cdot A_{\e2}^{-}(t)\\
&=-\int^0_{x_2^\e(t)-x_3^\e(t)}\rho_\e(y)\di y\left[p_1e^{x_1^\e(t)-x_2^\e(t)}(p_2+p_3e^{x_2^\e(t)-x_3^\e(t)})\right].
\end{aligned}
\end{equation}
Combining \eqref{eq:claim} and \eqref{eq:differ} gives 
\begin{align}\label{eq:decay}
\frac{\di}{\di t}(x_3^\e(t)-x_2^\e(t))<0,\quad t>t_1.
\end{align}
Next, we are going to use \eqref{eq:differ} to show that the following key estimate holds for any $t>t_1$:
\begin{align}\label{eq:strongclaim}
\lim_{\e\to0}\frac{x_3^\e(t)-x_2^\e(t)}{\e}=0.
\end{align}
\begin{proof}[Proof of \eqref{eq:strongclaim}]
We prove this by a contradiction argument. If \eqref{eq:strongclaim} does not hold, there exists $T_0>t_1$ and a subsequence $\e_n\to0$ such that
\[
\frac{x_3^{\e_n}(T_0)-x_2^{\e_n}(T_0)}{\e_n}\to a_0>0,\quad n\to+\infty.
\]
Consider $N>0$ big enough such that for any $t\in [t_1,T_0)$
\[
\frac{x_3^{\e_n}(t)-x_2^{\e_n}(t)}{\e_n}>\frac{x_3^{\e_n}(T_0)-x_2^{\e_n}(T_0)}{\e_n}\geq \frac{a_0}{2},\quad n\geq N.
\]
Consider in the time interval $[t_1,T_0]$. Since $|\frac{\di}{\di t}x_i^\e(t)|\leq \frac{1}{2}M_0^2$ ($M_0$ given by \eqref{eq:m0N}) for any $t>0$ and $\e>0$, we have $|x^\e_i(t)|\leq M_0^2T_0/2+|c_i|$ for any $0<t\leq T_0$ and $\e>0$.
For $n\geq N$ and $t\in[t_1,T_0]$, we have
\begin{multline}\label{eq:estimate}
\frac{\di}{\di t}(x_3^{\e_n}(t)-x_2^{\e_n}(t))=-\int^0_{x_2^{\e_n}(t)-x_3^{\e_n}(t)}\rho_{\e_n}(y)\di y\left[p_1e^{x_1^{\e_n}(t)-x_2^{\e_n}(t)}(p_2+p_3e^{x_2^{\e_n}(t)-x_3^{\e_n}(t)})\right]\\
\leq -\int_0^{a_0/2}\rho(y)\di y\left[p_1e^{-M_0^2T_0-|c_1|-|c_2|}\min\{p_2+p_3e^{c_2-c_3},~p_2+p_3\}\right]=:-A.
\end{multline}
Due to \eqref{eq:claim}, we know that $A$ is a positive constant independent of time $t\in[t_1,T_0]$ and $n\geq N$. 
Due to \eqref{eq:limit} and the fact that $x_2(t_1)=x_3(t_1)$, we can  choose $n_0\geq N$ big enough such that $x_3^{\e_{n_0}}(t_1)-x_2^{\e_{n_0}}(t_1)<A(T_0-t_1)$. Taking integration of \eqref{eq:estimate} in $[t_1,T_0]$ gives
\[
x_3^{\e_{n_0}}(T_0)-x_2^{\e_{n_0}}(T_0)\leq -A(T_0-t_1)+x_3^{\e_{n_0}}(t_1)-x_2^{\e_{n_0}}(t_1)<0,
\]
which is a contradiction with Proposition \ref{pro:nocollisionepsilon}. Hence \eqref{eq:strongclaim} holds.
\end{proof}

From \eqref{eq:strongclaim}, we have for any $t>t_1$ that
\[
h(\e,t):=\int_{0}^{x_3^\e(t)-x_2^\e(t)}\rho_\e(y)\di y=\int_{0}^{\frac{x_3^\e(t)-x_2^\e(t)}{\e}}\rho(y)\di y\to0,\quad \e\to0.
\]
Moreover, for each $\e>0$, $h(\e,\cdot)$ is a non-increasing function for $t>t_1$ since \eqref{eq:decay}.
Since $\rho_\e$ is an even function, from \eqref{eq:2e} and \eqref{eq:3e}, we obtain
\begin{align}\label{eq:2ee}
\frac{\di}{\di t}x_2^\e(t)=\frac{1}{2}A^{-}_{\e2}+C_1h(\e,t),\quad \frac{\di}{\di t}x_3^\e(t)=\frac{1}{2}A^{-}_{\e2}+C_2h(\e,t).
\end{align}
Here, for any bounded time interval $[0,T]$, $C_1$ and $C_2$ are positive constants depending on $T$ and $M_0$.
Combining \eqref{eq:123}, \eqref{eq:1e}, and \eqref{eq:2ee}, we have
\begin{align*}
\frac{\di}{\di t}\sum_{i=1}^3|x_i^\e(t)-x_i(t)|\leq& \sum_{i=1}^3\left|\frac{\di}{\di t}(x_i^\e(t)-x_i(t))\right|=\frac{3}{2}|A_{\e2}^{-}-A_2^-|+Ch(\e,t)\\
\leq& C\left(\sum_{i=1}^3|x_i^\e(t)-x_i(t)|+h(\e,t)\right)
\end{align*}
for some constant $C$ depending on $T$ and $M_0$.
By  the Gr\"onwall's inequality, we have
\[
\sum_{i=1}^3|x_i^\e(t)-x_i(t)|\leq \left(\sum_{i=1}^3|x_i^\e(t_1)-x_i(t_1)|+ C\int_{t_1}^th(\e,s)\di s\right)e^{C(t-t_1)}.
\]
Since $0\leq h(\e,t)\leq \frac{1}{2}$ for any $\e>0$ and $t>0$, by Lebesgue's dominate convergence theorem, we have 
\[
\lim_{\e\to0}\int_{t_1}^th(\e,s)\di s=\int_{t_1}^t\lim_{\e\to0}h(\e,s)\di s=0.
\]
Hence,
\[
\sum_{i=1}^3|x_i^\e(t)-x_i(t)|\to 0,\quad \e\to 0,~~t\in[0,T].
\]

\textbf{Proof of Case 3:} From \eqref{eq:deriva12}, we have
\[
\frac{\di}{\di t}(x_2(t)-x_1(t))=\frac{1}{2}A_2^+=\frac{1}{2}p_2p_3e^{x_2(t)-x_3(t)}+\frac{1}{2}p_1p_3e^{x_1(t)-x_3(t)}=\frac{1}{2}p_3e^{x_2(t)-x_3(t)}(p_2+p_1e^{x_1(t)-x_2(t)}).
\]
In the proof of Case 1, we have already assumed $p_1>0$ without loss of generality. Here, we assume $p_3>0$. Then, we must have $p_2<0$ and 
\begin{align}\label{eq:claim2}
p_2+p_1e^{-a}<0,\quad 0\leq a\leq c_2-c_1.
\end{align}
The proof of \eqref{eq:claim2} is the same as \eqref{eq:claim}. 

After the collision time $t_1$, we have
\[
x_i(t)\equiv x_1(t_1),~~t\geq t_1,~~i=1,2,3.
\]
For the regularized trajectories, similar to \eqref{eq:1e}, \eqref{eq:2e} and \eqref{eq:3e}, we have
\begin{multline}\label{eq:1e2}
\frac{\di }{\di t}x_1^\e(t)=\int_{x_1^\e(t)}^{x_2^\e(t)}\rho_\e(x_1^\e(t)-y)\di y A_{\e1}^{+}(t)+\int_{x_2^\e(t)}^{x_3^\e(t)}\rho_\e(x_1^\e(t)-y)\di y A_{\e2}^{+}(t)\\
=\int_{0}^{x_2^\e(t)-x_1^\e(t)}\rho_\e(y)\di y A_{\e2}^{-}(t)+\int_{x_2^\e(t)-x_1^\e(t)}^{x_3^\e(t)-x_1^\e(t)}\rho_\e(y)\di y A_{\e2}^{+}(t),
\end{multline}
\begin{equation}\label{eq:2e2}
\begin{aligned}
\frac{\di}{\di t}x_2^\e(t)=&\int_{x_2^\e(t)-\e}^{x_2^\e(t)+\e}\rho_\e(x_2^\e(t)-y)U_\e^N(y,t)\di y\\
=&\int_{x_1^\e(t)}^{x_2^\e(t)}\rho_\e(x_2^\e(t)-y)\di y\cdot A_{\e2}^{-}(t)+\int_{x_2^\e(t)}^{x_3^\e(t)}\rho_\e(x_2^\e(t)-y)\di y\cdot A_{\e2}^{+}(t)\\
=&\int^{x_2^\e(t)-x_1^e(t)}_{0}\rho_\e(y)\di yA_{\e2}^{-}(t)+\int_{0}^{x_3^\e(t)-x_2^\e(t)}\rho_\e(y)\di y\cdot A_{\e2}^{+}(t),
\end{aligned}
\end{equation}
and
\begin{align}\label{eq:3e2}
\frac{\di}{\di t}x_3^\e(t)=\int_{x_3^\e(t)-x_2^\e(t)}^{x_3^\e(t)-x_1^\e(t)}\rho_\e(y)\di y\cdot A_{\e2}^{-}(t)+\int_0^{x_3^\e(t)-x_2^\e(t)}\rho_\e(y)\di y\cdot A_{\e2}^{+}(t),
\end{align}
where we used  $A_{\e1}^{-}(t)=A_{\e3}^+(t)=0$. The above relations yield
\begin{align}\label{eq:dif1}
\frac{\di}{\di t}(x_2^\e(t)-x_1^\e(t))=\left(\int_{0}^{x_3^\e(t)-x_2^\e(t)}\rho_\e(y)\di y-\int_{x_2^\e(t)-x_1^\e(t)}^{x_3^\e(t)-x_1^\e(t)}\rho_\e(y)\di y\right)\cdot A_{\e2}^{+}(t)
\end{align}
and
\begin{align}\label{eq:dif2}
\frac{\di}{\di t}(x_3^\e(t)-x_2^\e(t))=\left(\int_{x_3^\e(t)-x_2^\e(t)}^{x_3^\e(t)-x_1^\e(t)}\rho_\e(y)\di y-\int^{x_2^\e(t)-x_1^\e(t)}_{0}\rho_\e(y)\di y\right)\cdot A_{\e2}^{-}(t)
\end{align}
Due to \eqref{eq:claim} and \eqref{eq:claim2}, we know that $A_{\e2}^{+}(t)<0$ and $A_{\e2}^{-}(t)>0$. Since $\rho_\e$ is an even function and increasing in $(-\e,0)$, we have
\[
\frac{\di}{\di t}(x_2^\e(t)-x_1^\e(t))<0,\quad \frac{\di}{\di t}(x_3^\e(t)-x_2^\e(t))<0.
\]
Using \eqref{eq:dif1} and \eqref{eq:dif2},  by similar method as proof of \eqref{eq:strongclaim}, we can obtain that for any $t>t_1$,
\begin{align}\label{eq:strongclaim2}
\lim_{\e\to0}\frac{x_3^\e(t)-x_2^\e(t)}{\e}=0~\textrm{ and }~ \lim_{\e\to0}\frac{x_2^\e(t)-x_1^\e(t)}{\e}=0,
\end{align}
or equivalently
\begin{align}\label{eq:strongclaim3}
	\lim_{\e\to0}\frac{x_3^\e(t)-x_1^\e(t)}{\e}=0.
\end{align}
Then the rest can be obtained similarly as Case 1.
\end{proof}
\begin{remark}
For $N=3$, there is only at most one time for collision between peakons. However, for the cases $N>3$, the collision between peakons is more complicated, and there are many different situations. Hence, it is very difficult to apply the method in Proposition \ref{pro:dispersionlimit} for  $N>3$. 
\end{remark}

\section{Splitting and non-uniqueness}\label{sec:splitting}
In this section, we will show that if the splitting of peakons happens, then the peakons are not unique, which might be different with the CH equation \eqref{eq:CH}.  We will provide some examples to illustrate this.
Consider one peakon starting from $(p,x_0)=(4,0)$. Then, one of the energy conservative solution is given by 
\begin{align}\label{eq:u1}
u_1(x,t)=pG(x-x_0).
\end{align}
Split the initial datum into three peakons $(p_1,x_1(0))=(5,0)$, $(p_2,x_2(0))=(-4,0)$, $(p_3,x_3(0))=(3,0)$, and then
\[
pG(x-x_0)=p_1G(x-x_1(0))+p_2G(x-x_2(0))+p_3G(x-x_3(0)).
\]
Accordingly, set $m_0(x)=p\delta(x-x_0)$ and $\hat{m}_0(x)=\sum_{i=1}^3p_i\delta(x-x_i(0))$. Then we actually have $m_0(x)=\hat{m}_0(x).$
Consider the following ODE system given by \eqref{eq:ODE} for $N=3$:
\begin{equation}
\left\{
\begin{aligned}
&\frac{\di}{\di t}x_1(t)=\frac{1}{2}p_1p_2e^{x_1(t)-x_2(t)}+\frac{1}{2}p_1p_3e^{x_1(t)-x_3(t)}=:A_1(t),\\
&\frac{\di}{\di t}x_2(t)=\frac{1}{2}p_1p_2e^{x_1(t)-x_2(t)}+\frac{1}{2}p_2p_3e^{x_2(t)-x_3(t)}+p_1p_3e^{x_1(t)-x_3(t)}=:A_2(t),\\
&\frac{\di}{\di t}x_3(t)=\frac{1}{2}p_1p_3e^{x_1(t)-x_3(t)}+\frac{1}{2}p_2p_3e^{x_2(t)-x_3(t)}=:A_3(t),
\end{aligned}
\right.
\end{equation}
with initial data $x_i(0)=0$ for $i=1,2,3$. Then, we have global trajectories $\{x_i(t)\}_{i=1}^3$. Moreover, notice that $A_1(0)<A_2(0)<A_3(0)$, which implies $x_1(t)<x_2(t)<x_3(t)$ for some small $T>0$ and $t\in (0,T)$. Let
\begin{align}\label{eq:u}
u(x,t)=p_1G(x-x_1(t))+p_2G(x-x_2(t))+p_3G(x-x_3(t)),\quad t\in(0,T).
\end{align}
Then, we have
\[
\frac{\di}{\di t}x_i(t)=u^2(x_i(t),t)-\overline{u^2_x}(x_i(t),t),\quad i=1,2,3.
\]
Since
\[
-\int_{\mathbb{R}}\varphi(x,0) m_0(\di x)=-\int_{\mathbb{R}}\varphi(x,0) \hat{m}_0(\di x)
\]
holds for any test function $\varphi\in C_c^\infty(\mathbb{R}\times[0,T))$,
using \eqref{useful},  we know  that $u_1$ and $u$ defined by \eqref{eq:u1} and \eqref{eq:u} are different conservative weak solutions to the mCH equation with the same initial data $pG(x-x_0)$ (or $m_0$ and $\hat{m}_0$) in $[0,T)$.

Note that this is different from the peakons for the CH equation \eqref{eq:CH}. For the CH equation, the $N$-peakon $u(x,t)=\sum_{i=1}^Np_i(t)e^{-|x-x_i(t)|}$ satisfies
\begin{equation}\label{eq:CHpeakons}
\left\{
\begin{aligned}
	&\frac{\di}{\di t}x_i(t)=\sum_{j=1}^Np_j(t)e^{-|x_i(t)-x_j(t)|},\\
	&\frac{\di}{\di t}p_i(t)=\sum_{j=1}^Np_i(t)p_j(t)\sgn(x_i(t)-x_j(t))e^{-|x_i(t)-x_j(t)|},
\end{aligned}
\right.
\end{equation}
which is a Hamiltonian system with Hamiltonian 
\[
\mathcal{H}_0(t)=\sum_{i,j=1}^Np_i(t)p_j(t)G(x_i(t)-x_j(t)).
\] 
Next, we are going to provide an example to show that even if we allow splitting of peakons, the above system will still give us the same solution. Consider one peakon $u(x,t)=p(t)e^{-|x-x(t)|}$ with $p(0)=c$ and $x(0)=x_0$. Using $\sgn(0)=0$, from \eqref{eq:CHpeakons}, we have
\begin{equation*}
\left\{
\begin{aligned}
&\frac{\di}{\di t}x(t)=p(t),\\
&\frac{\di}{\di t}p(t)=0,
\end{aligned}
\right.
\end{equation*}
which implies $u(x,t)=ce^{-|x-ct-x_0|}$. Assume that the peakon is split into two peakons: 
\[
u_2(x,t)=p_1(t)e^{-|x-x_1(t)|}+p_2(t)e^{-|x-x_2(t)|},
\]
and the initial data satisfy $P_1(0)+P_2(0)=c$ and $x_1(0)=x_2(0)=x_0$. We aim to show that the two peakons solution $u_2$ obtained by \eqref{eq:CHpeakons} is exactly the one peakon solution, i.e.,
\begin{align}\label{equiv}
u_2(x,t)=u(x,t)=ce^{-|x-ct-x_0|}.
\end{align}
Actually, from \eqref{eq:CHpeakons}, we have
\begin{equation*}
\left\{
\begin{aligned}
&\frac{\di}{\di t}x_1(t)=p_1(t)+p_2(t)e^{-|x_1(t)-x_2(t)|},\\
&\frac{\di}{\di t}x_2(t)=p_1(t)e^{-|x_1(t)-x_2(t)|}+p_2(t),\\
&\frac{\di}{\di t}p_1(t)=p_1(t)p_2(t)\sgn(x_1(t)-x_2(t))e^{-|x_1(t)-x_2(t)|},\\ 
&\frac{\di}{\di t}p_2(t)=p_1(t)p_2(t)\sgn(x_2(t)-x_1(t))e^{-|x_1(t)-x_2(t)|}.
\end{aligned}
\right.
\end{equation*}
Combining the last two equations, we obtain $\frac{\di}{\di t}(p_1(t)+p_2(t))=0$, which implies $p_1(t)+p_2(t)\equiv c$ for $t\geq0$. 
Then for $x_1$ and $x_2$, we have
\begin{equation*}
\left\{
\begin{aligned}
&\frac{\di}{\di t}x_1(t)=p_1(t)+(c-p_1(t))e^{-|x_1(t)-x_2(t)|},\\
&\frac{\di}{\di t}x_2(t)=p_1(t)e^{-|x_1(t)-x_2(t)|}+(c-p_1(t)).
\end{aligned}
\right.
\end{equation*}
By standard ODE theory, there exists a unique global solution $(x_1(t),x_2(t))$. 
Since $x_1(0)=x_2(0)=x_0$, one can check that $x_1(t)\equiv x_2(t)=ct+x_0$ is the unique solution. Therefore we have
\[
u_2(x,t)=p_1(t)e^{-|x-x_1(t)|}+p_2(t)e^{-|x-x_2(t)|}=ce^{-|x-ct-x_0|}=u(x,t).
\]

\noindent\textbf{Acknowledgements:}
This paper is supported by the National Natural Science Foundation grant 12101521 of China and the Start-up fund from The Hong Kong Polytechnic University with project number P0036186.

\appendix

\section{Proof of \eqref{eq:fact}}\label{fact}
\begin{proof}
Assume $x_0=0$ in \eqref{eq:fact}. Let $f:\mathbb{R}\to\mathbb{R}$ be a bounded piecewise continuous function with a jump discontinuity at $x=0$. Assume that the restrictions $f:[-1,0)\to\mathbb{R}$ and $f:(0,1]\to\mathbb{R}$ are continuous. Denote $f(0-):=\lim_{x\to0-}f(x)$ and $f(0+):=\lim_{x\to0+}f(x)$. Then we have
\begin{equation}
\begin{aligned}
(\rho_\e\ast f)(0)=\int_{\mathbb{R}}\rho_\e(0-y)f(y)\di y=\int_{-\e}^0\rho_\e(y)f(y)\di y+\int_{0}^\e\rho_\e(y)f(y)\di y,
\end{aligned}
\end{equation}
where we used the fact that $\rho_\e$ is an even function. By the continuity of $f$ at the right and left hand side of $x=0$, for any $\delta>0$, there exists $\e_0>0$ small enough such that for $|y|<\e<\e_0$, we have
\[
|f(y)-f(0-)|<\delta,\quad -\e<y<0,
\]
and
\[
|f(y)-f(0+)|<\delta,\quad 0<y<\e.
\]
Choosing $\e<\e_0$ and together with the fact that $\int_{0}^\e\rho_\e(y)\di y=\int_{-\e}^0\rho_\e(y)\di y=\frac{1}{2}$, we have
\begin{equation}
	\begin{aligned}
		|(\rho_\e\ast f)(0)-\bar{f}(0)|=&\left|\int_{-\e}^0\rho_\e(y)f(y)\di y+\int_{0}^\e\rho_\e(y)f(y)\di y-\frac{f(0-)+f(0+)}{2}\right|\\
		\leq &\int_{-\e}^0\rho_\e(y)|f(y)-f(0-)|\di y+\int_{0}^\e\rho_\e(y)|f(y)-f(0+)|\di y<\delta.
	\end{aligned}
\end{equation}
Therefore $\lim_{\e\to0}(\rho_\e\ast f)(0)=\bar{f}(0)$.

\end{proof}

\bibliographystyle{plain}
\bibliography{bibofPeak}

\end{document}